\documentclass[11pt]{amsart}
\usepackage[utf8]{inputenc}

\usepackage{overpic, amssymb, times, graphicx, tikz-cd, wrapfig}
\usepackage[all]{xy}
\usepackage[backref=page]{hyperref}
\usepackage{verbatim}
\hypersetup{
    colorlinks=true,
    linkcolor=blue,
    filecolor=blue,
    urlcolor=blue,
    citecolor=blue,
}
\DeclareMathOperator{\codim}{codim}
\DeclareMathOperator{\rank}{rank}
\DeclareMathOperator{\Fix}{Fix}
\DeclareMathOperator{\hot}{hot}
\DeclareMathOperator{\Morse}{Mo}
\DeclareMathOperator{\Sec}{Sec} % Space of sections of a bundle

\DeclareMathOperator{\Crit}{Crit}

\DeclareMathOperator{\Span}{span}

\DeclareMathOperator{\ind}{ind}

\DeclareMathOperator{\Symp}{Symp}
\DeclareMathOperator{\Int}{int}

\DeclareMathOperator{\Id}{Id}

\DeclareMathOperator{\Flow}{Flow}

\DeclareMathOperator{\PD}{PD}

\DeclareMathOperator{\sgn}{sgn}

\newcommand{\R}{\mathbb{R}}
\newcommand{\Z}{\mathbb{Z}}

\newcommand{\Q}{\mathbb{Q}}
\newcommand{\disk}{\mathbb{D}}
\newcommand{\grad}{\nabla}
\newcommand{\bigO}{\mathcal{O}}

\newcommand{\Cinfty}{\mathcal{C}^{\infty}}
\newcommand{\Rthree}{(\R^{3},\xi_{std})}

\newcommand{\Lie}{\mathcal{L}}

\newcommand{\sphere}{\mathbb{S}}
\newcommand{\Circle}{\sphere^{1}}

\newcommand{\half}{\frac{1}{2}}

\newcommand{\be}{\begin{enumerate}}
\newcommand{\ee}{\end{enumerate}}

\newcommand{\Mxi}{(M,\xi)}

\newcommand{\MxiDivSet}{(\divSet, \xiDivSet)}
\newcommand{\unknot}{\Gamma_{U}}
\newcommand{\MxiUnknot}{(\Gamma_{U}, \xi_{\Gamma_{U}})}

\newcommand{\norm}[1]{\left\lVert#1\right\rVert}

\newcommand{\annulus}{\mathbb{A}}
\newcommand{\stdSphere}[1]{\left(\sphere^{#1}, \xi_{std}\right)}

\newcommand{\Leg}{\Lambda}

\newcommand{\Stable}{\mathcal{S}}
\newcommand{\Unstable}{\mathcal{U}}
\newcommand{\framing}{\mathfrak{f}}
\newcommand{\ribbon}{\mathcal{R}}

\newcommand{\thicc}[1]{\pmb{#1}}
\newcommand{\divSet}{\Gamma}
\newcommand{\xiDivSet}{\xi_{\divSet}}

\newcommand{\Wdivisor}{V}
\newcommand{\WdivisorLift}{\widetilde{V}}
\newcommand{\betaW}{\beta_{W}}
\newcommand{\betaWD}{\beta_{\Wdivisor}}
\newcommand{\omegaW}{\omega_{W}}
\newcommand{\omegaWD}{\omega_{\Wdivisor}}
\newcommand{\betaWDLift}{\beta_{\WdivisorLift}}

\newtheorem{thm}{Theorem}[section]
\newtheorem{theorem}[thm]{Theorem}

\newtheorem{ex}[thm]{Example}
\newtheorem{assump}[thm]{Assumptions}

\newtheorem{defn}[thm]{Definition}

\newtheorem{lemma}[thm]{Lemma}

\newtheorem{rmk}[thm]{Remark}

\title{Stabilization of divisors in high-dimensional contact manifolds}
\author{Russell Avdek}
\date{\today}

\begin{document}
	
\begin{abstract}
A stabilization operation is defined for $\codim=2$ contact submanifolds $\Gamma$ in $\dim \geq 5$ contact manifolds $\Mxi$. The definition is such that (1) a given $\Mxi$ is overtwisted iff its standard contact unknot is stabilized and (2) contact stabilization preserves the formal contact isotopy class and intrinsic contact structure of a divisor. We prove that many $\Gamma$ are non-simple.
\end{abstract}

\maketitle
\numberwithin{equation}{subsection}
\setcounter{tocdepth}{1}
\tableofcontents

\section{Introduction}

This article concerns the tightness of contact manifolds $\Mxi$ of $\dim \geq 5$ as well as their $\codim=2$ contact submanifolds, $\divSet \subset \Mxi$. Such $\divSet$ are known as \emph{transverse links} when $\dim M =3$. We will call them \emph{contact divisors}, or simply \emph{divisors}, in general. The \emph{intrinsic contact structure} of a divisor $\Gamma$ is $\xiDivSet=\xi \cap T\Gamma$ and its \emph{formal contact isotopy class} is denoted $[\divSet]$. See \S \ref{Sec:Background} for background. Unless otherwise indicated, $\divSet$ and $M$ are assumed closed.

In \cite{CMP:OT} Casals, Murphy, and Presas prove the equivalence of various notions of overtwistedness for $\dim \geq 5$ contact manifolds. Here are a few such notions which are especially relevant for this article.

\begin{thm}[\cite{CMP:OT}]\label{Thm:CMP}
Let $\Mxi$ be a $\dim \geq 5$ contact manifold. Then the following are equivalent:
\be
\item $\Mxi$ is overtwisted in the sense of Borman, Eliashberg, and Murphy \cite{BEM:Overtwisted}.
\item The standard Legendrian unknot $\Leg_{U} \subset \Mxi$ is loose in the sense of Murphy \cite{Murphy:Loose}.
\item $\Mxi$ can be expressed as a contact anti-surgery along a loose Legendrian sphere in some $(M', \xi')$.\footnote{Contact anti-surgeries are also known as contact $+1$ surgeries. See \S \ref{Sec:WeinsteinSurgery}.}
\ee
\end{thm}

Applying the Legendrian isotopy classification of Murphy's \cite{Murphy:Loose}, we can replace the ``loose Legendrians'' in the above statements with the ``stabilized Legendrians'' of Ekholm, Etnyre, and Sullivan \cite{EES:Nonisotopic}. With this slight modification, the statement of the theorem agrees with the equivalence of analogous formulations of overtwistedness for contact $3$-manifolds. However, there remains a formulation of overtwistedness in $\dim=3$ without an established high-dimensional analogue, cf. \cite{Etnyre:KnotNotes}.

\begin{theorem}\label{Thm:Dim3Stab}
A $\Mxi$ of $\dim = 3$ is overtwisted iff its standard contact unknot is stabilized.
\end{theorem}

\begin{figure}[h]
\begin{overpic}[scale=.4]{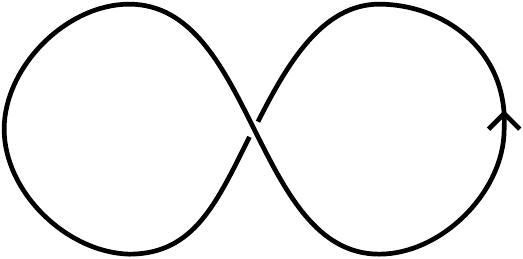}
\end{overpic}
\caption{The $\unknot$ of $\dim=1$  in the ($x, t$) projection of $\R^{3}_{x, y, t}$ with contact form $dt - ydx$.}
\label{Fig:TUnknot}
\end{figure}

When $\dim M=3$, the standard contact unknot is the well-known transverse unknot of self-linking number $-1$ shown in Figure \ref{Fig:TUnknot}. In any odd dimension the standard contact unknot is generally defined as follows: Let
\begin{equation*}
\left(\disk^{2n+1}\subset \R_{t}\times \R^{n}_{x}\times \R^{n}_{y}, \xi_{std}\right), \quad \xi_{std} = \ker \left(dt + \beta_{std}\right), \quad \beta_{std} = \sum (x_{i}dy_{i} - y_{i}dy_{i})
\end{equation*}
be a Darboux disk. The standard contact unknot in the Darboux disk is
\begin{equation*}
\unknot = \{ t = 0 \}\cap \partial \disk^{2n+1} \subset \left(\disk^{2n+1}, \xi_{std}\right).
\end{equation*}
For $\unknot$ inside of a general $\Mxi$, work within a Darboux disk in $\Mxi$. Intrinsically $\Gamma_{U}$ is the standard contact sphere of dimension $2n-1$, given by the boundary of the standard Liouville disk,
\begin{equation*}
\MxiUnknot \simeq \stdSphere{2n-1} = \partial \left(\disk^{2n}, \beta_{std}\right).
\end{equation*}

\subsection{Stabilization of divisors and overtwistedness in $\dim \geq 5$}

The purpose of this article is to define stabilization of high dimensional contact divisors so that the following theorem holds.

\begin{thm}\label{Thm:MainOT}
A $\dim \geq 5$ contact manifold $\Mxi$ is overtwisted iff $\unknot \subset \Mxi$ is stabilized. Stabilization preserves the intrinsic contact structure and formal contact isotopy class of a $\codim=2$ contact submanifold inside of a $\dim \geq 5$ contact manifold.
\end{thm}

The first statement is the high-dimensional analogue of Theorem \ref{Thm:Dim3Stab} and the second is analogous to Legendrian stabilization preserving formal Legendrian isotopy classes \cite{EES:Nonisotopic, Murphy:Loose}. By contrast, formal isotopy classes of transverse and Legendrian links in $\Mxi$ of $\dim=3$ are modified by stabilization. There are many (a priori different) ways to stabilize divisors $\Gamma$ in $\dim \geq 5$ contact manifolds. As shown in \S \ref{Sec:StandardStabilization}, certain stabilizations can always be performed within an arbitrarily small neighborhood of a $\Gamma$.

The fact that $\unknot$ is a stabilization if $\Mxi$ is overtwisted follows immediately from Cordona and Presas' parametric $h$-principle for contact divisors with overtwisted complements \cite[Corollary 3.5]{CP:TransverseHprinciple}. Since $\unknot$ is contained in a Darboux ball, the overtwistedness of $\Mxi$ implies the overtwistedness of $(M \setminus \unknot, \xi)$ and \cite{CP:TransverseHprinciple} applies. Likewise whenever a $\divSet$ has an overtwisted complement then it is contact isotopic to each of its stabilizations.

\subsection{Non-simplicity of contact divisors}

In addition to relating divisors to overtwistedness of the $\Mxi$ which contain them, we are interested in understanding the geography of $\Gamma$ in a fixed formal isotopy class $[\divSet]$. In addition to \cite{CP:TransverseHprinciple}, recent results of Casals, Pancholi, and Presas \cite{CPP:Whitney}, Honda and Huang \cite[Corollary 1.3.6]{HH:Convex}, Lazarev \cite{Lazarev:Maximal}, and Pancholi and Pandit \cite{PP:IsoContact} address the existence of (genuine) contact submanifolds $\divSet$ within a fixed (formal) class $[\divSet]$.

\begin{defn}
We say that $\divSet$ is \emph{non-simple} if there exist a $\divSet' \subset \Mxi$ with $[\divSet]=[\divSet']$, $(\divSet', \xi_{\divSet'})$ isomorphic to $\MxiDivSet$, and for which $\divSet'$ is not contact isotopic to $\divSet$ within $\Mxi$. A class $[\divSet]$ is \emph{non-simple} is it has a non-simple contact representative.
\end{defn}

The first examples of non-simple contact submanifolds in $\Mxi$ of $\dim=3$ are due to Birman and Menasco \cite{BirmanMenasco}. For further references on the low-dimensional case, see \cite{Etnyre:KnotNotes} or the more recent \cite{CasalsEtnyre}.

As will be evident from the techniques used throughout, this article is largely inspired by Casals and Etnyre's \cite{CasalsEtnyre} in which the first examples of non-simple contact divisors in $\Mxi$ of $\dim \geq 5$ are constructed. See Zhou's \cite{Zhou:Embeddings} and C\^{o}t\'{e} and Fauteux-Chapleau's \cite{CFC:HCSubmanifold} for further examples. The divisors of \cite{CasalsEtnyre, Zhou:Embeddings} are implicitly unit cotangent bundles of spheres, $(\sphere^{\ast}\sphere^{n}, \xi_{can} = \ker \beta_{can})$, with $\beta_{can} = \sum p_{i}dq_{i}$ the canonical $1$-form. Using our notion of stabilization, we prove that many contact divisors are non-simple. For example, the following is immediate from Theorem \ref{Thm:MainOT} and the existence of contact stabilizations.

\begin{thm}\label{Thm:UnknotNonsimple}
If $\Mxi$ is tight, then $\unknot$ is non-simple.
\end{thm}

We similarly prove that many other contact divisors are non-simple by establishing that they are not stabilizations using symplectic cobordism arguments.

\begin{thm}\label{Thm:HypersurfaceNonSimple}
Suppose that $\divSet \subset \Mxi$ bounds a Liouville hypersurface in some $\Mxi$ of $\dim \geq 5$. If $\Mxi$ is weakly fillable or has non-vanishing contact homology (over $\Q$), then $\divSet$ is not stabilized and is therefore non-simple.
\end{thm}

Of course, the hypothesis on $\Mxi$ is satisfied for the most important cases of the standard sphere $\Mxi = \stdSphere{2n+1}$ and unit cotangent bundles of smooth manifolds. We conjecture that the theorem holds under the more general assumption that $\Mxi$ is tight. Theorem \ref{Thm:HypersurfaceNonSimple} is analogous to the fact that a stabilized transverse link in a tight contact $3$-manifold cannot bound a Liouville hypersurface.\footnote{This fact follows Eliashberg's Thurston-Bennequin inequality  \cite{Eliash:TBInequality} together with the effect of contact stabilization on self-intersection numbers. See \cite{Etnyre:KnotNotes}.}

Contact submanifolds $\divSet$ bounding Liouville hypersurfaces are central objects in contact topology. Examples include the unknots $\unknot$, contact push-offs of Legendrians \cite{CasalsEtnyre} (see \S \ref{Sec:HypersurfaceDivisor}), and bindings of supporting open books \cite{Giroux:ContactOB,BHH:OB, Sackel:Handle}. Bindings include Milnor links \cite{Milnor:Singular} and $\divSet$ in prequantization spaces associated to Donaldson divisors in closed, integral symplectic manifolds \cite{CDVK:BoothbyWang}. Independent of the tightness or overtwistedness of $\Mxi$, bindings of open books are never stabilized.

\begin{thm}\label{Thm:BindingNonStabilized}
The binding of a supporting open book for a $\Mxi$ of $\dim \geq 5$ is not stabilized. Consequently every such $\Mxi$ has infinitely many non-simple $\divSet$.
\end{thm}

When $\Mxi$ is overtwisted, the non-simplicity of bindings of supporting open books has previously been established in \cite[Theorem 12.7]{CFC:HCSubmanifold}.

\begin{comment}
Observe that the preceding non-simplicity results all concern contact divisors admitting Seifert framings. This restriction is not necessary for our techniques to apply. Suppose that $\divSet \subset \Mxi$ is a contact divisor and $(W, \betaW)$ is a Liouville filling of $\Mxi$ containing a $\codim=2$ Liouville submanifold $(\Wdivisor, \betaWD=\betaW|_{T\Wdivisor})$ which fills $\divSet$. Then $(\Wdivisor, W)$ is a \emph{Liouville filling of the triple  $(\divSet, M, \xi)$}. Using the lifting operation of \S \ref{Sec:HypersurfaceDivisor}, this generalizes $\divSet$ bounding a Liouville submanifold in a Liouville fillable $\Mxi$. See also Example \ref{Ex:PreQFilling}.

\begin{thm}\label{Thm:FillingDivisor}
If $(\divSet, M, \xi)$ is Liouville fillable, then $\divSet$ is not stabilized and so is non-simple.
\end{thm}
\end{comment}

All the non-simplicity results mentioned above concern contact divisors $\divSet$ whose implicit contact structures $\xiDivSet$ are symplectically fillable. The technique does not apply to high dimensional divisors which are implicitly overtwisted.

\begin{thm}\label{Thm:OTDivisor}
Suppose that $\Mxi$ has $\dim \geq 7$ and $\divSet \subset \Mxi$ is such that $\MxiDivSet$ is overtwisted. Then $\divSet$ is stabilized.
\end{thm}

We expect the result should extend to the $\dim M =5$ case as well.

\subsection{Outline of the article}

In \S \ref{Sec:Background} we cover background material required for our proofs. Apart from standard reference material we describe some specific models for Weinstein handles.

The model Weinstein handles are used in \S \ref{Sec:AmbientHandles} to define \emph{ambient Weinstein surgeries and anti-surgeries}. These are contact-topological versions of the usual handle-body modification of smooth $\codim=2$ submanifolds as in Rolfsen's \cite[\S 11.A]{Rolfsen}. For transverse links in contact $3$-manifolds, these handle attachments locally correspond to crossing changes in a front projection. See \S \ref{Sec:LowDimHandles}. Contact stabilization is then defined using ambient Weinstein handle attachments. After the definition is provided, basic examples and properties are covered. In \S \ref{Sec:AntiStabilization}, we interpret the construction of non-simple $\divSet$ from \cite{CasalsEtnyre} in the language of ambient Weinstein surgeries, describing an \emph{anti-stabilization} operation on general contact divisors in high dimensions.

In \S \ref{Sec:NormalAndFiberSum} we study normal bundles of contact divisors and review the contact fiber sum. This is prerequisite for \S \ref{Sec:FiberSum} which details a fiber sum cobordism construction combining Geiges' \cite{Geiges:Branched} and Gompf's \cite{Gompf:Sum}. In \S \ref{Sec:MainProofs} we use properties of our fiber sum cobordisms to prove the theorems stated in this introduction.

\subsection{A remark on our methodology}

In \cite{CasalsEtnyre}, contact divisors are distinguished by studying Weinstein cobordisms associated to their branched covers. We briefly explain why we use fiber sums instead of branched covers to prove Theorem \ref{Thm:MainOT}.

It can be shown that if $\unknot'$ is a stabilization of $\unknot \subset \stdSphere{2n+1}$, then a double branched cover of $\stdSphere{2n+1}$ over $\divSet'$ is overtwisted. This is enough to establish that $\unknot'$ is not contact isotopic to $\unknot$, since a double branched covering of the standard sphere over $\unknot$ is itself the standard sphere.\footnote{In most lectures supporting this work, we have described this branched covering strategy to establish the non-simplicity of $\unknot \subset \stdSphere{2n+1}$, since the proof is much shorter than that of Theorem \ref{Thm:MainOT}. The details will appear in a future article.}

Now suppose we seek to apply the same strategy to establish that $\unknot$ is not stabilized in a general, tight $\Mxi$ of $\dim \geq 5$. The previously mentioned techniques will show that a double branched covering of $\Mxi$ over a stabilized $\unknot'$ is overtwisted. So to distinguish $\unknot$ from $\unknot'$ we would like to show that a double branched cover of $\Mxi$ over $\unknot$ is tight. It follows from \cite[Theorem 1.20]{Avdek:Liouville} that this double branched cover is the contact-connected sum $\#2 \Mxi$ of $\Mxi$ with itself. Our proof would then be complete if it was established that $\# 2 \Mxi$ is tight. However, in contrast with the $\dim M=3$ case \cite{Colin:TightGluing, DG:TightDecomposition, Honda:Gluing}, it is not currently known if connected sums of tight contact manifolds of $\dim \geq 5$ are always tight.

\subsection{Acknowledgments}

\begin{wrapfigure}{r}{0.07\textwidth}
\centering
\includegraphics[width=.07\textwidth]{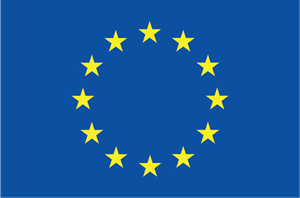}
\end{wrapfigure}

We thank Cofund MathInGreaterParis and the Fondation Mathématique Jacques Hadamard for supporting us as a member of the Laboratoire de Math\'{e}matiques d'Orsay at Universit\'{e} Paris-Saclay. This project has received funding from the European Union’s Horizon 2020 research and innovation programme under the Marie Skłodowska-Curie grant agreement No 101034255.

This project grew out of conversations with Ko Honda, for which we are very grateful. We're very thankful to Wenyuan Li for pointing out that Theorem \ref{Thm:OTDivisor} follows from Eliashberg's \cite{Eliash:WeinsteinRevisited}. We also thank Joseph Breen, Robert Cardona, Roger Casals, Austin Christian, John Etnyre, Fabio Gironella, Yang Huang, and Zhengyi Zhou for interesting discussions and feedback on earlier versions of this article. Updated versions of this article have benefited enormously from feedback at seminars and conferences held at the American Institute of Mathematics, Institut Henri Poincar\'{e}, Nantes Universit\'{e}, Université Grenoble Alpes, and the University of Iowa.

\section{Prerequisites}\label{Sec:Background}

This section covers background material required for our proofs. In addition to the material covered here, we will also use some basic properties of convex hypersurfaces as well as supporting open books and their stabilizations in \S \ref{Sec:MainProofs}. For a reference, see \cite{BHH:OB}. In \S \ref{Sec:HypersurfaceProof}, we also use some basic properties of contact homology -- namely the existence of cobordism maps and twisted coefficient systems, cf. \cite{Bourgeois:ContactIntro}.

\subsection{Legendrian stabilization}\label{Sec:LegStabilization}

We briefly review formal Legendrian isotopy \cite{EM:IntroHPrinciple, Murphy:Loose} and the Legendrian stabilization of \cite{EES:Nonisotopic}. Let $\Mxi$ be a $\dim=2n+1$ contact manifold and let $\Leg$ be a smooth $\dim=n$ manifold. A \emph{formal Legendrian embedding} is a $S \in [0, 1]$ family of commutative diagrams
\begin{equation*}
\begin{tikzcd}
T\Leg \arrow[r, "\Phi_{S}"]\arrow[d] & TM \arrow[d]\\
\Leg \arrow[r, "\phi"] & M
\end{tikzcd}
\end{equation*}
with the vertical arrows being the canonical projections and the $\Phi_{S}$ depending continuously on $S$. The following properties are required to be satisfied:
\be
\item $\phi$ is a smooth embedding with $\Phi_{0} = T\phi$.
\item $\Phi_{S}$ is a smooth fiberwise-injective bundle map covering $\phi$ for each $S$.
\item $\Phi_{1}(T_{q}\Lambda)$ is a Lagrangian subspace of $\xi_{\phi(q)}$ for each $q\in \Lambda$.
\ee
A \emph{formal Legendrian isotopy} is a $T\in [0, 1]$ family $(\Phi_{S, T}, \phi_{T})$ of formal Legendrian embeddings.

Now we define Legendrian stabilization. The standard contact structure $(\R^{2n+1}, \xi_{std})$ is
\begin{equation*}
\R^{2n+1} = \R_{t} \times \R^{n}_{p}\times \R^{n}_{q}, \quad \xi_{std} = \ker\left(dt + \sum p_{i}dq_{i}\right).
\end{equation*}
It is the $1$-jet space of $\R^{n}_{q}$. We'll assume $n\geq 2$. Let $f(\rho): \R \rightarrow \R$ be an increasing function such that $\exists \epsilon > 0$ for which $f(\rho) = -1$ for $\rho < -\epsilon$ and $f(\rho) = 1$ for $\rho > \epsilon$. Consider the Legendrian embedding
\begin{equation*}
\R^{n}_{q} \rightarrow \R^{2n+1}, \quad (q_{1}, \cdots, q_{n}) \mapsto \left(f(q_{1}^{3}), \left(-\frac{3}{2}q_{1}\frac{\partial f}{\partial \rho}\left(q_{1}^{3}\right), 0, \dots, 0\right), (q_{1}^{2}, q_{2}, \dots, q_{n})\right)
\end{equation*}
whose image will be denoted $\Lambda_{C}$. The front projection ($\to \R^{n+1}_{t,q}$) of $\Lambda_{C}$ is singular along the set $\{ t = q_{1} = 0\} \subset \Lambda_{C}$ which we call the \emph{cusp}. Observe that along the set $q_{1} > \epsilon^{2}$, that the Lagrangian projection ($\to \R^{2n}_{p,q}$) for $\Lambda_{C}$ is a two-sheeted covering of its image, with the sheets given by the $1$-jets of the constant functions $t=\pm 1$.

Let $\gamma \subset \{ q_{1} > \epsilon^{2}\} \subset \R^{n}_{q}$ be an embedded circle with a tubular neighborhood $N_{\gamma} = \gamma \times \disk^{n-1}_{x}$. We view $N_{\gamma}$ as being contained in the lower sheet $\{ t=-1, q_{1} > \epsilon^{2}\}$ of $\Leg_{C}$. Take a bump function $g(\rho)$ for which $g(0) = 1+\delta$ and $g(\rho) = -1$ for $\rho > \delta$, with $\delta$ a small positive constant. Then set $g(q) = |x|^{2}$ inside of $N_{\gamma}$ and $g = -1$ elsewhere. Define $\Lambda_{\gamma}$ to be Legendrian given by replacing the lower sheet $\{ t=-1, q_{1} > \epsilon^{2}\}$ with the $1$-jet lift $(g(q), -dg, q)$ of the function $g$. For any such $\gamma, g$ we say that $\Lambda_{\gamma}$ is a \emph{stabilization} of $\Lambda_{C}$. See Figure \ref{Fig:StabilizationFront} which mimics \cite[Figure 2.1]{EM:Caps}.

\begin{figure}[h]
\begin{overpic}[scale=.5]{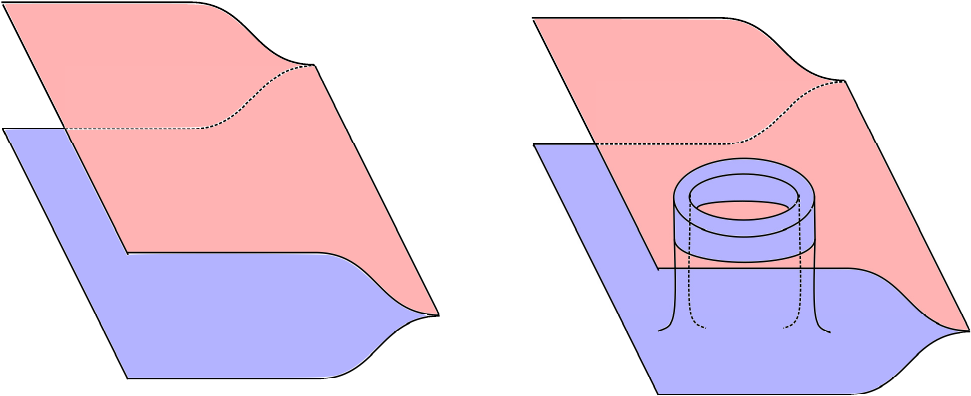}
\end{overpic}
\caption{On the left, a cusp neighborhood in the $n=2$ front projection with the upper sheet shaded red and the lower sheet shaded blue. On the right, a stabilization is performed along a circle in the lower sheet.}
\label{Fig:StabilizationFront}
\end{figure}

Now we apply the preceding constructions to Legendrians in arbitrary contact manifolds. Provided a Legendrian $\Leg$ in some $\Mxi$ and a disk $\disk^{2n+1} \subset \Mxi$ such that the pair $(\disk^{2n+1}, \disk^{2n+1} \cap \Lambda)$ is isomorphic to a disk in $(\R^{2n+1}, \Leg)$ taking $0 \in \disk^{2n+1}$ to a point on the cusp, we say that $(\disk^{2n+1}, \Leg)$ is a \emph{cusp chart} for $\Leg$. Since contact disks centered about points in Legendrians are all locally isomorphic, every $q \in \Leg$ is centered about some cusp chart.

\begin{defn}
A Legendrian $\Leg \subset \Mxi$ is a \emph{stabilization of a Legendrian $\Leg' \subset \Mxi$} \cite{EES:Nonisotopic} if $\Leg'$ has a cusp chart $(\disk^{2n+1}, \Leg')$ for which
\be
\item $\Leg$ and $\Leg'$ coincide outside of $\disk^{2n+1}$ and
\item within $\disk^{2n+1}$, $\Leg$ is obtained from stabilizing $\Leg'$.
\ee
A Legendrian $\Leg \subset \Mxi$ is \emph{loose} \cite{Murphy:Loose} if it is a stabilization of some $\Leg'$.
\end{defn}

\begin{lemma}[\cite{EES:Nonisotopic}]
If $\Leg$ is a stabilization of $\Leg'$, then $\Leg$ and $\Leg'$ are formally Legendrian isotopic.
\end{lemma}

\subsection{Formal contact structures and isotopies}

\begin{defn}
Let $M$ be a $\dim=2n+1$ manifold equipped with a hyperplane sub-bundle $\eta \subset TM$ and a $2$-form $\omega$. We say that $(\eta, \omega)$ is a \emph{formal contact structure on $M$} if $\omega$ is fiber-wise symplectic on $\eta$.
\end{defn}

Of course a contact form $\alpha$ for a contact manifold $\Mxi$ determines a formal contact structure $(\xi, d\alpha)$ on $M$ which, modulo deformation, is independent of $\alpha$. Now we address submanifolds and homotopies.

\begin{defn}\label{Def:FormalContact}
Let $(\eta_{M}, \omega_{M})$ be a formal contact structure on a $(2n+1)$-dimensional manifold $M$ and let $(\eta_{\Gamma}, \omega_{\Gamma})$ be a formal contact structure on a $(2n-1)$-dimensional manifold $\Gamma$. A \emph{formal contact embedding} of $(\Gamma, \eta_{\Gamma}, \omega_{\Gamma})$ into $(M, \eta_{M}, \omega_{M})$ is a $S \in [0, 1]$ continuous family of commutative diagrams
\begin{equation*}
\begin{tikzcd}
T\Gamma \arrow[r, "\Phi_{S}"]\arrow[d] & TM \arrow[d]\\
\Gamma \arrow[r, "\phi"] & M
\end{tikzcd}
\end{equation*}
which satisfy the following properties:
\be
\item $\phi$ is an embedding with tangent map $T\phi = \Phi_{0}$,
\item $\Phi_{S}$ is a fiber-wise injective bundle map for all $S$, and
\item $\Phi_{1}$ is a conformally symplectic map from $(\eta_{\Gamma}, \omega_{\Gamma})$ to $(\eta_{M}, \omega_{M})$.
\ee
A pair $(\phi_{T}, \Phi_{S, T}), T=0, 1$ of formal contact embeddings are \emph{formally contact isotopic} if they are connected by a $T\in [0, 1]$ family of formal contact embeddings. The image of a formal contact embedding is a \emph{formal contact submanifold}.
\end{defn}

\subsection{Some symplectic manifolds}\label{Sec:SympManifolds}

An \emph{exact symplectic manifold} is a pair $(W, \beta_{W})$ consisting of a manifold $W$ (possibly with corners) and a \emph{Liouville form} $\betaW \in \Omega^{1}(W)$, meaning that $d\betaW$ is symplectic. Its \emph{Liouville vector field}, $X_{\betaW}$, is defined by
\begin{equation*}
d\betaW(X_{\betaW}, \ast) = \betaW \iff \Lie_{X_{\betaW}}\betaW = \betaW.
\end{equation*}
A \emph{Liouville cobordism} is an exact symplectic manifold with smooth boundary such that $X_{\betaW}$ is transverse to $\partial W$. Write $\partial^{+}W$ ($\partial^{-}W$) for the portion of $\partial W$ along which $X_{\betaW}$ points out of (respectively, into) $W$. Then $\partial^{+}W$ is the \emph{convex boundary} and $\partial^{-}W$ is the \emph{concave boundary}. Both $\partial^{\pm}W$ are contact manifolds when equipped with the contact forms $\betaW|_{T\partial^{\pm}W}$. When $\partial^{-}W=\emptyset$, we say that $(W, \betaW)$ is a \emph{Liouville filling} of its convex boundary.

We briefly mention strong and weak symplectic cobordisms which we won't need to address until \S \ref{Sec:HypersurfaceProof}. See \cite{MNW13} for further details. A \emph{weak symplectic cobordism} $(W, \omega)$ is a symplectic manifold with boundary $\partial W = \partial^{-}W \sqcup \partial^{+}W$ such that each $\partial^{\pm}W$ has a collar neighborhood $[0, 1]_{s} \times \partial^{\pm}W \subset W$ along which $\omega = e^{s}\alpha_{\pm} + \omega_{\pm}$ where $\alpha_{\pm}$ is a contact form on $\partial^{\pm}W$ and $\omega_{\pm} \in \Omega^{2}(\partial^{\pm}W)$ is closed. When $\omega_{\pm}=0$, we say that $(W, \omega)$ is a \emph{strong symplectic cobordism}, or simply, a symplectic cobordism. When $\partial^{-}W=\emptyset$, we say that $(W, \omega)$ is a \emph{weak symplectic filling}. A \emph{strong symplectic filling} is a weak symplectic filling with $\omega_{+}=0$.

\begin{rmk}
Our definition of weak cobordisms is implicitly using \cite[Lemma 1.10]{MNW13}, which provides a Weinstein neighborhood type theorem for weak cobordisms as defined in \cite[Definition 4]{MNW13}.
\end{rmk}

A \emph{Weinstein cobordism} is a triple $(W, \betaW, F_{W})$ for which $(W, \betaW)$ is a Liouville cobordism and $F_{W}: W \rightarrow \R$ is a smooth function such that
\be
\item the zeros of $\betaW$ agree with the critical points of $F_{W}$ and 
\item $df(X_{\betaW}) > 0$ away from critical points.
\ee
This implies that $dF_{W}$ is non-zero along $\partial W$. We say that $F_{W}$ is a \emph{Lyapunov function} for $\betaW$. If $F_{W}$ is Morse and all its critical points have Morse index $< \half \dim W$, we say that $(W, \betaW, F_{W})$ is \emph{subcritical}. The canonical modern reference for Weinstein manifolds is Cieliebak and Eliashberg's \cite{SteinToWeinstein}. A Liouville cobordism $(W, \betaW)$ is \emph{Weinstein} if there exists an $F_{W}$ for which $(W, \betaW, F_{W})$ is Weinstein.

\begin{assump}
The $F_{W}$ on Weinstein manifolds are assumed Morse unless otherwise indicated.
\end{assump}

Critical point sets will be denoted $\Crit(F_{W}) \subset W$. Provided a $(W, \betaW, F_{W})$, write
\begin{equation*}
W_{s} = \{ F_{W} = s \}, \quad W_{\leq s} = \{ F_{W} \leq s\}, \quad W_{\geq s} = \{ F_{W} \geq s \},
\end{equation*}
so that when $s$ is a regular value of $F_{W}$, $W_{s}$ is a contact manifold and the $W_{\leq s}, W_{\geq s}$ are Weinstein cobordisms.

The following results concerning deformations of Weinstein structures is a restatement of \cite[Lemma 12.1]{SteinToWeinstein}. See \cite{Lazarev:Simple} for recent results on deformations of Weinstein structures.

\begin{lemma}\label{Lemma:BetaFunctionMultiple}
Provided a Liouville cobordism $(W, \betaW)$ and $f \in \Cinfty(W)$, $e^{f}\betaW$ is a Liouville form on $W$ iff $1 + df(X_{\betaW}) > 0$ in which case the new Liouville field is computed
\begin{equation*}
X_{e^{f}\betaW} = (1 + df(X_{\betaW}))^{-1}X_{\betaW}.
\end{equation*}
If in addition $F_{W}$ is a Lyapunov function giving the cobordism a Weinstein structure, and $e^{f}\betaW$ is Liouville, then $(W, e^{f}\betaW, F_{W})$ is Weinstein.
\end{lemma}

\begin{comment}
\begin{lemma}\label{Lemma:SelfIndexing}
Given a $(W, \betaW, F_{W})$ there exists a $T\in [0, 1]$ family of Lyapunov functions $F_{W, T}$ for $X_{\betaW}$ for which $F_{W} = F_{W, 0}$ and $F_{W, 1}$ is \emph{self-indexing}. That is, for all $\zeta \in \Crit(F_{W}) = \Crit(F_{W, 1})$,
\begin{equation*}
F_{W, 1}(\zeta) = \ind_{\Morse}(F_{W, 1}, \zeta).
\end{equation*}
\end{lemma}
\end{comment}

\subsection{(Un)stable manifolds and contact (anti)surgery}\label{Sec:WeinsteinSurgery}

Let $\zeta$ be a zero of a vector field $X$ defined on a manifold. We recall that stable and unstable manifolds associated to the pair $(\zeta, X)$ are defined
\begin{equation*}
\begin{aligned}
\Stable(\zeta, X) = \left\{ z\ :\ \lim_{T\rightarrow \infty}\Flow^{T}_{X}(z) = \zeta\right\},\quad
\Unstable(\zeta, X) = \left\{ z\ :\ \lim_{T\rightarrow -\infty}\Flow^{T}_{X}(z) = \zeta\right\}.
\end{aligned}
\end{equation*}
For the following results, see \cite[\S 11.3-11.4]{SteinToWeinstein}.

\begin{lemma}\label{Lemma:GeneralBetaVanishing}
Let $(W, \betaW)$ be a Liouville manifold with $\zeta$ a non-degenerate zero of $\betaW$. Then $\betaW$ is zero along $T\Stable(\zeta, X_{\betaW})$ implying that $\Stable(\zeta, X_{\betaW})$ is isotropic in $W$. If $\dim T_{\zeta}\Unstable(\zeta, X_{\betaW}) = \half \dim W$, then $\betaW$ is also zero along $T\Unstable(\zeta, X_{\betaW})$, implying that $\Unstable(\zeta, X_{\betaW})$ is Lagrangian.

Moreover if $F_{W}$ is Lyapunov for $(W, \betaW)$, then for each regular value $s$ of $F_{W}$ and non-degenerate zero $\zeta \in W$ of $\betaW$, the spheres
\begin{equation*}
\Leg^{\Stable}_{s, \zeta} = \Stable(\zeta, X_{\betaW}) \cap W_{s}
\end{equation*} are isotropic in the contact manifolds $W_{s}$. When the Morse index is $\ind_{\Morse}(\zeta, \Crit(F_{W})) = \half \dim W$, then both $\Lambda_{s, \zeta}^{\Stable}$ and
\begin{equation*}
\Lambda_{s, \zeta}^{\Unstable} = \Unstable(\zeta, X_{\betaW}) \cap W_{s}
\end{equation*}
are Legendrian for each regular value $s$.
\end{lemma}

When $\zeta \in \Crit(F_{W})$ with $s = F_{W}(\zeta)$ then for $\epsilon > 0$ small, $(W_{\leq s + \epsilon}, \betaW, F_{W})$ can be seen as the result of a \emph{Weinstein surgery} \cite{Weinstein:Handles} determined by attaching a Weinstein handle along the isotropic sphere $\Stable(\zeta, X_{\betaW}) \cap W_{s - \epsilon}$ in the convex boundary of $W_{\leq s - \epsilon}$. Specific models for Weinstein handles will be provided in the next subsection.

If for each $\ind_{\Morse}(\zeta, F_{W}) = \half \dim W$ critical point, the Legendrian sphere $\Leg^{\Stable}_{F_{W}(\zeta)-\epsilon, \zeta}$ is loose, we say that the cobordism is \emph{flexible}. Flexible cobordisms include subcritical cobordisms as a proper subset.

When $\ind_{\Morse}(\zeta, F_{W}) = \half \dim W$ we also say that $(W_{\geq s - \epsilon}, \betaW, F_{W})$ is obtained by \emph{Weinstein anti-surgery} along the Legendrian sphere $\Lambda_{s, \zeta}^{\Unstable}$ in the concave boundary of $(W_{\geq s+ \epsilon}, \betaW, F_{W})$. So $(W_{\geq s - \epsilon}, \betaW, F_{W})$ is obtained from $(W_{\geq s - \epsilon}, \betaW, F_{W})$ by removing a neighborhood of a $\dim=n$ unstable manifold, which is a Lagrangian slice disk for its boundary. See \cite{CET:PlusOneFilling, DLMW:Antisurgery} for recent results on contact anti-surgery which employ the ``cutting out a Lagrangian disk'' point of view. The ``anti'' versions of subcritical surgeries are \emph{coisotropic surgeries} described in \cite{CE:Caps}.

In the critical $\ind_{\Morse}(\zeta, F_{W})=\half \dim W$ case we say that
\be
\item the contact manifold $W_{s + \epsilon}$ is the result of a \emph{contact surgery} along $\Leg^{\Stable}_{s-\epsilon, \zeta}$ and
\item the contact manifolds $W_{s - \epsilon}$ is the result of \emph{contact anti-surgery} along $\Leg^{\Unstable}_{s + \epsilon, \zeta}$.
\ee
Contact surgeries and anti-surgeries can be defined without reference to Weinstein structures, using convex gluing and Dehn twists \cite{Avdek:Liouville, DG:Surgery}. These are also respectively known as contact $-1$ surgeries and contact $+1$ surgeries due to the fact that in the $\dim W = 4$ case, they are smoothly Dehn surgeries along Legendrian knots whose surgery coefficients are $\pm 1$ when computed with respect to the contact framing. We recommend \cite{OS:SurgeryBook} for an introduction to the low dimensional case. Contact surgeries along Legendrian are often called ``Legendrian surgeries'' in the literature and we've heard the terms ``upside-down Legendrian surgeries'' and ``Legendrian anti-surgeries'' used for what we are calling contact anti-surgeries.

\begin{rmk}
Contact surgery as defined in \cite{Avdek:Liouville} (which generalizes \cite{DG:Surgery} to high dimension) requires framing data as input: The Legendrian sphere along which contact surgery is performed must be identified with the boundary of a disk, and there is in general more than one such identification in higher dimensions by the existence of exotic smooth homotopy spheres given by gluing high-dimensional disks along their boundaries \cite{KM:ExoticSpheres}. We disregard such framing data as all of our contact surgeries will be the result of Weinstein surgeries (anti-surgeries) whose stable (resp. unstable) disks provide the necessary framing.
\end{rmk}

We'll need the following results concerning Weinstein surgeries and overtwistedness, the first being \cite[Proposition 2.12]{CMP:OT}.

\begin{lemma}\label{Lemma:SubcriticalPreservesOT}
If $(W, \betaW, F_{W})$ is a flexible Weinstein cobordism whose concave boundary is overtwisted, then its convex boundary is overtwisted as well.
\end{lemma}

\begin{lemma}\label{Lemma:PlusOneSurgeryPreservesOT}
The result of a contact anti-surgery along a Legendrian sphere $\Leg$ in an overtwisted $\Mxi$ is overtwisted as well.
\end{lemma}

\begin{proof}
This is immediate from Theorem \ref{Thm:CMP}. Since the Legendrian unknot in $\Mxi$ is loose and $\Leg$ can be viewed as a connected sum of itself with the unknot, $\Leg$ is loose as well. So we are performing contact anti-surgery on a loose Legendrian sphere.
\end{proof}

\subsection{Model weinstein handles}\label{Sec:HandleAbstract}

While Weinstein handles appear naturally about critical points in Weinstein manifolds, we want to have models on hand so that they can be used in \S \ref{Sec:AmbientHandles} to modify contact submanifolds by \emph{ambient surgeries} in the style of \cite[\S 11.A]{Rolfsen}. Here we develop model Weinstein handles with variable geometry and ``sharp edges'', making their boundaries branched manifolds. Varying the geometries of the handles will then correspond to application of contact isotopies to the inputs and outputs of surgeries. Similar constructions in which Weinstein handles are carefully shaped (for dynamical applications) appear in \cite{Avdek:Dynamics, Cieliebak:SubcriticalSH, Ekholm:SurgeryCurves, HT:Chord}.

\begin{comment}
A $\dim=2n, \ind=k\leq n$ Weinstein handle $(H_{k}, \beta_{k})$ consists of a $\dim=2n$ manifold with corners $H_{k}$ and a $1$-form $\beta_{k}$ satisfying the following properties:
\be
\item $d\beta_{k}$ is symplectic on $H_{k}$.
\item The Liouville vector field $X_{\beta_{k}}$ has a Morse Lyapunov function $F_{k}$ having a single critical point $\zeta$ of index $k$.
\item $\partial H_{k}$ is the closure of closed subsets $\partial^{\pm}H_{k}$.
\item $\partial^{-}H_{k}$ is the set along which $X_{\beta_{k}}$ points into $H_{k}$ and is a disk neighborhood of the $\dim=k$ sphere $\Leg^{\Stable} = \Stable(\zeta, X_{\beta_{k}}) \cap \partial^{-}H_{k}$.
\item $\partial^{+}H_{k}$ is the set along which $X_{\beta_{k}}$ points out of $H_{k}$ and is a disk neighborhood of the $\dim=2n-k$ sphere $\Leg^{\Unstable} = \Unstable(\zeta, X_{\beta_{k}})\cap \partial^{+}H_{k}$.
\item The tangent spaces of the boundaries of the $\partial^{\pm}H_{k}$ agree.
\ee
It follows that $\beta_{k}$ is a contact form on each of the $H_{k}$, within which $\Leg^{\Stable}$ is isotropic and $\Leg^{\Unstable}$ is coisotropic.

We sketch a model construction of a Weinstein handle. 
\end{comment}

Our model for a $\dim=2n, \ind=k$ handle will be $(H_{k}, \beta_{k}, F_{k})$ with $H_{k} = H_{k}(\vec{\epsilon})$ a subset of $\R^{2n}$ to be defined shortly. We define $\beta_{k}$ and $F_{k}$ on $\R^{2n}$ and calculate
\begin{equation*}
\begin{gathered}
\beta_{k} = \half\sum_{1}^{n}(x_{i}dy_{i} - y_{i}dx_{i}) + d\left(\sum_{1}^{k}x_{j}y_{j} \right), \quad F_{k} = \sum_{1}^{k} \left(x_{i}^{2} - y_{i}^{2}\right) + \sum_{k+1}^{n} \left(x_{j}^{2} + y_{j}^{2}\right),\\
\implies d\beta_{k} = \sum_{1}^{n} dx_{i}\wedge dy_{j}, \quad 2X_{\beta_{k}} = \sum_{1}^{k} \left(3x_{i}\partial_{x_{i}} - y_{i}\partial_{y_{i}}\right) + \sum_{k+1}^{n} \left(x_{j}\partial_{x_{j}} + y_{j}\partial_{y_{j}}\right).
\end{gathered}
\end{equation*}
There is a single critical point at $\zeta = 0$. Let's rewrite the coordinates as
\begin{equation*}
\vec{x}=(x_{1}, \dots, x_{n}), \quad \vec{y}_{1,k}=(y_{1},\dots,y_{k}),\quad  \vec{y}_{k+1,n}=(y_{k+1},\dots,y_{n})
\end{equation*}
so that $\R^{2n} = \R^{n}_{\vec{x}}\times \R^{k}_{\vec{y}_{1,k}}\times \R^{n-k}_{\vec{y}_{k+1,n}}$. Then
\begin{equation*}
\begin{gathered}
\Stable(\zeta, X_{\beta_{k}}) = \{ 0 \}_{\vec{x}} \times \R^{k}_{\vec{y}_{1,k}} \times \{0\}_{y_{k+1,n}},\\
\Unstable(\zeta, X_{\beta_{k}}) = \R^{n}_{\vec{x}} \times \{0\}_{y_{1,k}} \times \R^{n-k}_{\vec{y}_{k+1,n}}
\end{gathered}
\end{equation*}
so we can view
\begin{equation*}
\R^{2n} = \Stable(\zeta, X_{\beta_{k}})_{v^{\Stable}}  \times \Unstable(\zeta, X_{\beta_{k}})_{v^{\Unstable}}, \quad v^{\Stable} = \vec{y}_{1, k}, \quad v^{\Unstable}=(\vec{x},\vec{y}_{k+1, n}).
\end{equation*}

As a first approximation to $H_{k}$, consider $H^{\square}_{k} = \disk^{\Stable}_{v^{\Stable}} \times \disk^{\Unstable}_{v^{\Unstable}} \subset \R^{n}$ where $\disk^{\Stable}$ is a disk in the stable manifold and $\disk^{\Unstable}$ is a disk in the unstable manifold. For now we'll disregard the radii of the disks. Then 
\begin{equation*}
\partial^{-}H^{\square}_{k} = (\partial \disk^{\Stable}) \times \disk^{\Unstable}, \quad \partial^{+}H^{\square}_{k} = \disk^{\Stable} \times (\partial \disk^{\Unstable})
\end{equation*}
and the corner of the manifold is $\partial\disk^{\Stable} \times \partial \disk^{\Unstable}$ where the tangent spaces of the $\partial^{\pm}H^{\square}_{k}$ are orthogonal.

\begin{figure}[h]
\begin{overpic}[scale=.4]{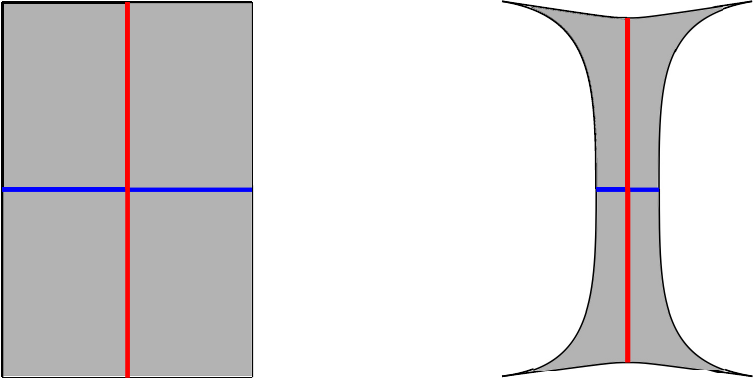}
\end{overpic}
\caption{The manifolds $H_{k}^{\square}$ and $H_{k}$ appear on the left and right, respectively. Here and throughout, we color stable manifolds blue and unstable manifolds red following the mnemonic ``the blue disk is below the critical point''.}
\label{Fig:PointyCorners}
\end{figure}

The corner can be ``sharpened'' so that the tangent spaces of the $\partial^{\pm}H_{k}$ agree along the corner $\partial^{+}H_{k} \cap \partial^{-}H_{k}$ to obtain a handle $H_{k} \subset H^{\square}_{k}$. See Figure \ref{Fig:PointyCorners}. We provide a family of models with variable geometry.

Let $B^{\pm}_{\epsilon} = B^{\pm}_{\epsilon}(\rho), \rho \in \R_{\geq 0}, \epsilon >0$ be smooth functions characterized by the following properties:
\be
\item $B^{-}(\rho) = 0$ for $\rho < \epsilon$ and $B^{-}(\rho) < \rho$ for $\rho < 3\epsilon$ along which the function is increasing,
\item $B^{+}(\rho) = 2\epsilon$ for $\rho < \epsilon$, $B^{+}$ is increasing and $< \rho$ when $\rho < 3\epsilon$, and
\item $B^{\pm}(\rho)=\rho$ for $\rho \geq 3\epsilon$.
\ee
For a pair of positive constants $\vec{\epsilon} = (\epsilon^{\Stable}, \epsilon^{\Unstable})$ define $\epsilon^{\max} = \max\{ \epsilon^{\Stable}, \epsilon^{\Unstable} \}$ and
\begin{equation*}
\begin{gathered}
H_{k}(\vec{\epsilon}) = \left\{ B^{+}_{\epsilon^{\Stable}}(\norm{v^{\Stable}}) \geq B^{-}_{\epsilon^{\Stable}}(\norm{v^{\Unstable}}), \quad B^{+}_{\epsilon^{\Unstable}}(\norm{v^{\Unstable}}) \geq B^{-}_{\epsilon^{\Unstable}}(\norm{v^{\Stable}}), \quad \norm{v^{\Stable}}, \norm{v^{\Unstable}} \leq 3\epsilon^{\max}\right\},\\
\partial^{-}H_{k}(\vec{\epsilon}) = \left\{ B^{+}_{\epsilon^{\Stable}}(\norm{v^{\Stable}}) = B^{-}_{\epsilon^{\Stable}}(\norm{v^{\Unstable}})\right\},\\
\partial^{+}H_{k}(\vec{\epsilon}) = \left\{ B^{+}_{\epsilon^{\Unstable}}(\norm{v^{\Unstable}}) = B^{-}_{\epsilon^{\Unstable}}(\norm{v^{\Stable}})\right\},\\
\partial^{-}H_{k}(\vec{\epsilon}) \cap \partial^{+}H_{k}(\vec{\epsilon}) = \left\{ \norm{v^{\Stable}} = \norm{v^{\Stable}} = 3\epsilon^{\max}\right\}.
\end{gathered}
\end{equation*}
Then the desired transversality and corner tangency conditions are satisfied for $H_{k}(\vec{\epsilon})$.

The $\vec{\epsilon}$ will be omitted from notation when it is unimportant. We will use it to either shrink the handle (by making $\epsilon^{\max}$ small) or make it ``long and thin'' as in the right hand side of Figure \ref{Fig:PointyCorners} by fixing $\epsilon^{\Unstable}$ and choosing $\epsilon^{\Stable} \ll \epsilon^{\Unstable}$.

\subsection{Liouville hypersurfaces and symplectic divisor cobordisms}\label{Sec:HypersurfaceDivisor}

Let $(\Wdivisor, \betaWD)$ be an exact symplectic manifold of $\dim=2n$ and let $\Mxi$ be a contact manifold of $\dim=2n+1$. A \emph{Liouville embedding} is an inclusion $\Wdivisor \subset M$ such that there is a subset of $\Mxi$ of the form $N(\Wdivisor) = [-\epsilon, \epsilon]_{t} \times \Wdivisor \subset M$ within which $\Wdivisor = \{t=0\}$ and $\xi = \ker (dt + \beta)$.

The image of a Liouville embedding is an \emph{exact symplectic hypersurface}. If $(\Wdivisor, \betaWD)$ is a Liouville cobordism, the image is a \emph{Liouville hypersurface}. If additionally $X_{\betaWD}$ has a Lyapunov function, then $\Wdivisor \subset \Mxi$ is a \emph{Weinstein hypersurface}. The convex and concave boundaries $\partial^{\pm}V$ of a Liouville hypersurface are disjoint contact divisors in $\Mxi$. In \cite{Avdek:Liouville} Liouville hypersurfaces are required to have $\partial^{-}V = \emptyset$.

\begin{ex}
Every Legendrian $\Leg \subset \Mxi$ has a neighborhood of the form
\begin{equation}\label{Eq:OneJet}
N(\Leg) = [-\epsilon, \epsilon] \times \disk^{\ast}\Leg, \quad \xi = \ker(dt + \beta_{can}).
\end{equation}
The \emph{contact push-off} of $\Leg$ is $\divSet_{\Leg} = \partial \{ t = 0\}$. Implicitly, $\MxiDivSet$ is the unit cotangent bundle of $\Leg$.
\end{ex}

Let $(W, \omegaW)$ be a symplectic cobordism containing a $\codim=2$ symplectic submanifold $(\Wdivisor, \omegaWD = \omegaW|_{T\Wdivisor})$. We say that $\Wdivisor$ is a \emph{symplectic divisor} (or simply, a divisor) of $(W, \omega)$. If $\partial^{\pm}\Wdivisor$ is contained in $\partial^{\pm}W$ with $\Wdivisor$ transverse to $\partial^{\pm}W$ then $\Wdivisor$ is a \emph{divisor cobordism}. If $(W, \betaW)$ and $(\Wdivisor, \betaWD=\beta_{W}|_{T\Wdivisor})$ are Liouville, then $\Wdivisor$ is a \emph{Liouville divisor cobordism}. These are called exact relative cobordisms in \cite{CFC:HCSubmanifold}. If $\Wdivisor$ is a Liouville divisor cobordism and $\betaW$ admits a Lyapunov function $F_{W}$ restricting to a Lyapunov function $F_{\Wdivisor} = F_{W}|_{\Wdivisor}$, then $\Wdivisor$ is a \emph{Weinstein divisor cobordism}.

Provided a Liouville hypersurface $\Wdivisor \subset \Mxi$ take a function $F_{0}: \Wdivisor \rightarrow [-C, C]$ for which $F_{0}^{-1}(\pm C) = \partial^{\pm}\Wdivisor$ for some $C > 0$. For a fixed $F_{0}$ and $\delta > 0$ sufficiently small the graph
\begin{equation*}
\WdivisorLift = \{ (F_{\Wdivisor}(z), z)\ : z\in \Wdivisor\} \subset [-\delta C, \delta C] \times M
\end{equation*}
of $F_{V} = \delta F_{0}$ inherits a Liouville form
\begin{equation*}
\betaWDLift = e^{s}\alpha|_{T\WdivisorLift} = e^{F_{\Wdivisor}}\betaWD
\end{equation*}
from the symplectization by Lemma \ref{Lemma:BetaFunctionMultiple}. We say that $\WdivisorLift \subset [-C, C] \times M$ is a \emph{lift of $\Wdivisor$}.

For Weinstein hypersurfaces, we always take the lifting function to be a Lyapunov function $F_{\Wdivisor}$ for $\betaWD$. Then $(\WdivisorLift, \betaWDLift)$ is Weinstein by Lemma \ref{Lemma:BetaFunctionMultiple}. By the same lemma,
\begin{equation}\label{Eq:BetaLift}
X_{\WdivisorLift} = \left(1 + dF_{V}(X_{\betaWD})\right)^{-1}X_{\betaWD}.
\end{equation}
Moreover, when $\betaWD$ is Weinstein then the above assumptions imply that
\begin{equation*}
F_{\WdivisorLift} = s|_{\WdivisorLift} = F_{\Wdivisor}
\end{equation*}
is Lyapunov for $\betaWDLift$. For a lift of a Weinstein hypersurface the zeros of $\betaWDLift$ are of the form
\begin{equation*}
\widetilde{\zeta} = (F_{\Wdivisor}(\zeta), \zeta) \subset \WdivisorLift \subset [-C, C]_{s} \times M
\end{equation*}
for zeros $\zeta$ of $\betaWD$. Likewise
\begin{equation*}
\Stable(\zeta, X_{\betaWD}) = \pi_{M}(\Stable(\widetilde{\zeta}, X_{\betaWDLift})), \quad \Unstable(\zeta, X_{\betaWD}) = \pi_{M}(\Unstable(\widetilde{\zeta}, X_{\betaWDLift}))
\end{equation*}
since the Liouville vector fields for $X_{\Wdivisor}$ and $X_{\WdivisorLift}$ are parallel.

\section{Ambient surgery and contact stabilization}\label{Sec:Stabilization}

In this section we define stabilization of contact divisors $\divSet$ inside of $\Mxi$ of $\dim M = 2n+1 \geq 5$. The stabilization construction requires a definition of ambient Weinstein surgeries as a prerequisite, worked out in \S \ref{Sec:AmbientHandles}. The $\dim M=3$ case is described diagrammatically in \S \ref{Sec:LowDimHandles}. In \S \ref{Sec:divisorCobordism} we see that such surgeries, as well as certain surgeries applied to $\Mxi$, determine Weinstein divisor cobordisms between contact submanifolds. The ambient surgery construction is analogous to the Lagrangian cobordism construction of Ekholm, Honda, and K\'{a}lm\'{a}n \cite{EHK:LegCobordism}.

After defining stabilization of contact divisors in \S \ref{Sec:StabilizationDef}, basic examples are provided in \S \ref{Sec:StandardStabilization}, and basic properties are outlined in \S \ref{Sec:StabilizationProperties}. In \S \ref{Sec:AntiStabilization}, we define an \emph{anti-stabilization operation} on contact divisors, which interprets the constructions of \cite{CasalsEtnyre} in the language of ambient Weinstein surgeries. Finally, in \S \ref{Sec:HighDimLowDim} curious readers can see the effect of our high-dimensional stabilization on transverse links in contact $3$-manifolds.

\subsection{Ambient Weinstein surgeries}\label{Sec:AmbientHandles}

Let $(H_{k}, \beta_{k}, F_{k})$ be a $\dim=2n, \ind=k$ model Weinstein handle as described in \S \ref{Sec:HandleAbstract}. The image of Liouville embedding of a Weinstein handle into $\Mxi$ is an \emph{ambient Weinstein handle}, denoted
\begin{equation*}
(H_{k}, \beta_{k}) \subset \Mxi.
\end{equation*}
Since Weinstein handles have a single critical point, we can speak of their \emph{stable and unstable manifolds} unambiguously, which will be denoted
\begin{equation*}
\Stable_{k}, \Unstable_{k} \subset H_{k},
\end{equation*}
respectively. By Lemma \ref{Lemma:GeneralBetaVanishing}, $\Stable_{k}$ is isotropic in $\Mxi$ with $\partial\Stable_{k}$ isotropic in the open contact submanifold $\partial^{-}H_{k}$ of $\Mxi$. Likewise when $k=n$, both $\Stable_{k}$ and $\Unstable_{k}$ are Legendrian disks in $\Mxi$  and $\partial \Unstable_{k}$ is a Legendrian sphere in $\partial^{+}H_{n}$.

\begin{figure}[h]
\begin{overpic}[scale=.4]{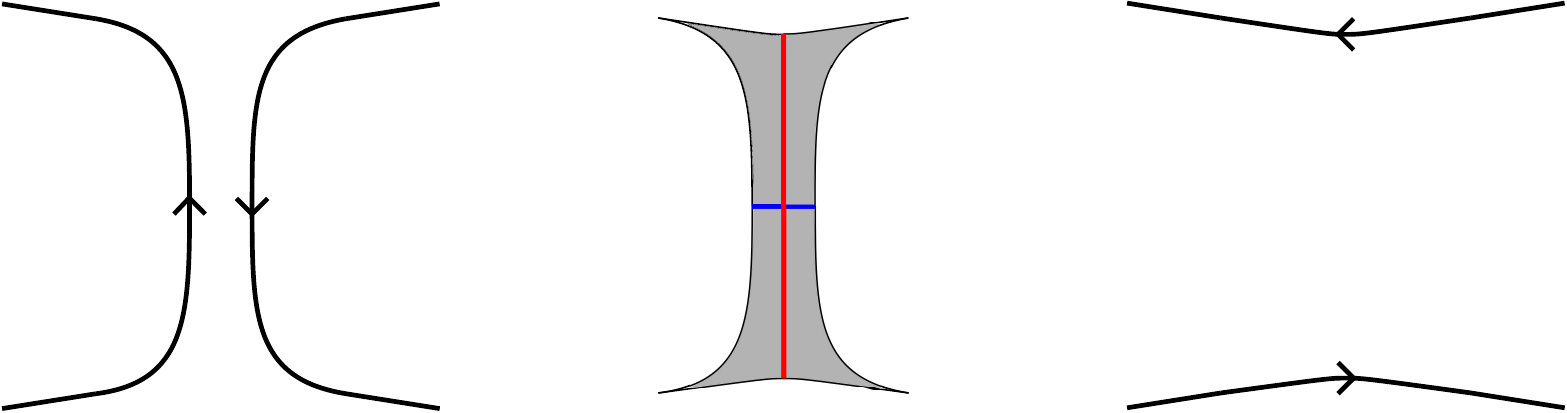}
\end{overpic}
\caption{An ambient Weinstein surgery turns $\Gamma^{-}$ (left) into $\Gamma^{+}$ (right), or vice versa by an anti-surgery.}
\label{Fig:AmbientSurgery}
\end{figure}

Suppose we have a contact divisor $\divSet^{-} \subset \Mxi$ and an ambient Weinstein handle $(H_{k}, \beta_{k}) \subset \Mxi$ for which $H_{k} \cap \divSet^{-} = \partial^{-}H_{k}$. Then $(H_{k}, \beta_{k}, \Gamma^{-})$ is \emph{ambient Weinstein surgery data}. It follows that
\begin{equation*}
\Gamma^{+} = (\Gamma^{-} \setminus \partial^{-}H_{k}) \cup \partial^{+}H_{k}
\end{equation*}
is a contact divisor of $\Mxi$ is the result of an \emph{ambient Weinstein surgery} of $\Gamma^{-}$ along $H_{k}$. 

Likewise if we have a divisor $\Gamma^{+}$ and a $(H_{n}, \beta_{n}) \subset \Mxi$ for which $\partial^{+}H_{n} = H_{n}\cap \divSet^{+}$, then $(H_{n}, \beta_{n}, \Gamma^{+})$ is a \emph{Weinstein anti-surgery data} and
\begin{equation*}
\Gamma^{-} = (\Gamma^{+} \setminus \partial^{+}H_{n}) \cup \partial^{-}H_{n}
\end{equation*}
is the result of an \emph{ambient Weinstein anti-surgery} of $\Gamma^{+}$ along $H_{n}$. Note that ``anti-surgery'' is reserved for $\ind =n$ handles. In this critical index case, $(\divSet^{+}, \xi_{\divSet^{+}})$ is (intrinsically) the result of a contact surgery along the Legendrian sphere $\partial \Stable_{n} \subset \divSet^{-}$. Likewise $(\divSet^{-}, \xi_{\divSet^{-}})$ is the result of a contact anti-surgery along $\partial \Unstable_{n} \subset \divSet^{+}$.

\subsection{Ambient handle attachments applied to transverse links}\label{Sec:LowDimHandles}

Let's work out how ambient Weinstein surgeries apply to transverse links in contact $3$-manifolds. In this case the only handles to consider have $\ind=0$ and $1$. To build the model handles in this case, we work with the contact form $\alpha = dt - ydx$ on $M = \R^{3}_{x, y, t}$ allowing us to draw pictures in the front ($x,t$) projection. Again, we refer to \cite{Etnyre:KnotNotes} for background.

For the $\ind=0$ handle attachment, our handle is a Weinstein disk, $\disk$. So starting from a $\divSet^{-}$, the effect of the handle attachment yields a $\divSet^{+}$ which is the disjoint union of $\divSet^{-}$ and a standard transverse unknot $\unknot$ bounding a $\disk \subset M$ which is disjoint from $\divSet^{-}$. This can easily be drawn in the front as in Figure \ref{Fig:TUnknot}.

\begin{figure}[h]
\begin{overpic}[scale=.35]{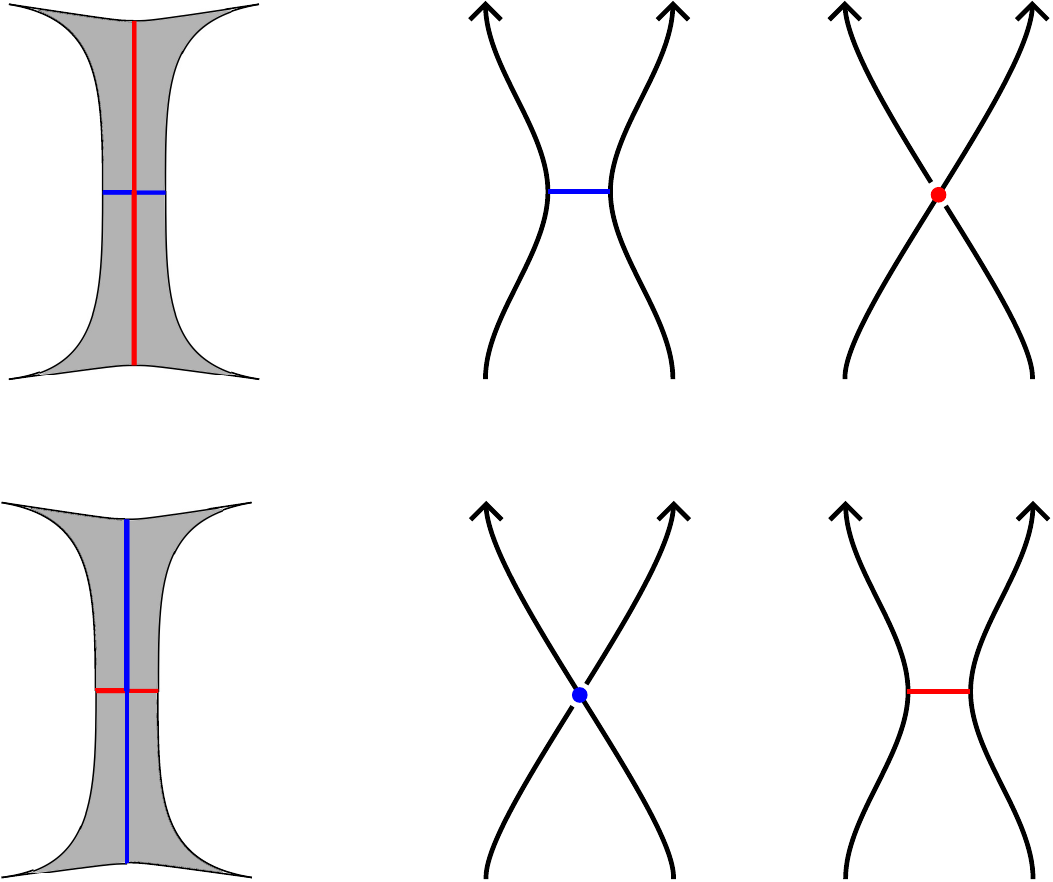}
\put(-10, 62){$H_{-}$}
\put(-10, 15){$H_{+}$}
\put(40, 62){$\divSet^{-}$}
\put(40, 15){$\divSet^{-}$}
\put(75, 62){$\divSet^{+}$}
\put(75, 15){$\divSet^{+}$}
\end{overpic}
\caption{On the top row we attach the handle $H_{-}$ (left in the Lagrangian $xy$ projection) to a $\divSet^{-}$ (center in the front $xt$ projection) to obtain a $\divSet^{+}$ (right in the front projection). As always, the stable manifolds of the handle are colored blue and unstable manifolds are colored red. The bottom row works out the analogous picture for a $H_{+}$ handle attachment.}
\label{Fig:LowDimHandles}
\end{figure}
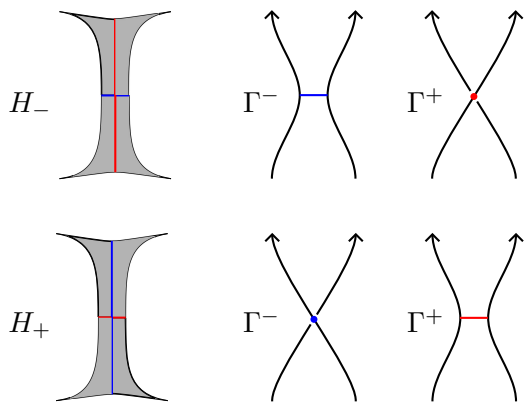

For the $\ind=1$ handle attachments, we consider two local models. Let $H_{x, y} \subset \R^{2}_{x, y}$ be a subset with sharp corners centered about $(0, 0)$ as shown in the left column of Figure \ref{Fig:LowDimHandles}. Let $f_{\pm} = \pm 2xy$ and let $H_{\pm} \subset \R^{3}$ be the $t=f_{\pm}$ graph over $H_{\pm}$. The Liouville vector field for $\beta_{\pm} = \alpha|_{TH_{\pm}}$ is computed in $x,y$ coordinates as
\begin{equation*}
X_{\beta_{\pm}} = \pm 2x\partial_{x} + (1 \mp 2)y\partial_{y}
\end{equation*}
having a single zero at $(0, 0, 0) \in \R^{3}$. As always, write $\Stable$ and $\Unstable$ for the stable and unstable manifolds of this zero. The smooth boundary components of $H_{\pm}$ lift to transverse link segments as shown in Figure \ref{Fig:LowDimHandles}.

For $H_{-}$ (in the top row of the figure), $\Stable$ is contained in the $x$ axis with boundary on $\divSet^{-}$ and $\Unstable$ is contained in the $y$ axis, projecting to a positive crossing of $\divSet^{+}$ in the font projection. So the effect of the surgery is to add a positive crossing.

For $H_{+}$ (in the bottom row of the figure), $\Stable$ is contained in the $y$ axis, projecting to negative crossing of $\divSet^{+}$, and $\Unstable$ is contained in the $x$ axis with boundary on $\divSet^{+}$. So the effect of the surgery is to resolve a negative crossing.

\subsection{Divisor cobordisms from ambient Weinstein surgeries}\label{Sec:divisorCobordism}

The following lemma says that ambient contact surgeries determine Weinstein divisor cobordisms between contact divisors.

\begin{lemma}
Suppose that $\divSet^{-} \subset \Mxi$ is a contact divisor in a $\dim = 2n+1$ contact manifold and that $(H_{k}, \beta_{k}, \divSet^{-})$ is Weinstein surgery data determining a contact divisor $\divSet^{+}$. Then there is a Weinstein divisor cobordism $\Wdivisor$ in a finite symplectization $(W, \betaW) = ([-1, 1]_{s}\times M, e^{s}\alpha)$ of $\Mxi$ such that
\be
\item $\partial^{\pm}\Wdivisor = \divSet^{\pm}$ and $\pi_{M}\Wdivisor = \divSet^{-} \cup H_{k} \cup \divSet^{+}$,
\item $\pi_{M}$ restricted to each of $\pi_{M}^{-1}(\divSet^{\pm})$ is a diffeomorphism onto its image,
\item $\betaWD$ has a unique zero within the lift of $H_{k}$, having the same index as the unique zero of $\beta_{k}$.
\ee
\end{lemma}

\begin{proof}
We assume that our contact form $\alpha$ for $\Mxi$ takes the form $dt + \beta_{k}$ on a standard neighborhood of the handle. We lift the $\Int(H_{k}) \subset M$ to the symplectization as the graph as a function $s = F:\Int(H_{k}) \rightarrow (0, 1)$, take the closure of this graph in $[0, 1]\times M$, and then append a vertical piece
\begin{equation*}
[-1, 1] \times (\divSet^{-} \setminus (\partial^{-}H_{k})) = [-1, 1] \times (\divSet^{+} \setminus (\partial^{+}H_{k})).
\end{equation*}
The union of the closure of the graph and the vertical piece gives the cobordism $\Wdivisor$. We require that $F$ satisfies the following conditions:
\be
\item $\lim_{z \rightarrow \partial^{\pm}H_{k}}F(z) = \pm 1$.
\item $F(\zeta) = dF(\zeta) = 0$ where $\zeta$ is the unique zero of $\beta_{k}$.
\item $dF\left(\grad\norm{v^{\Stable}}^{2}\right) \leq 0 \leq dF\left(\grad\norm{v^{\Unstable}}^{2}\right)$.
\item $\lim_{z \rightarrow \partial^{-}}dF\left(\grad\norm{v^{\Stable}}^{2}\right) =-\infty$ and $\lim_{z \rightarrow \partial^{+}}dF\left(\grad\norm{v^{\Unstable}}^{2}\right) =\infty$.
\ee
The last condition ensures that the closure of the graph of $F$ with the vertical piece is smooth. The remaining conditions ensure that $\betaWD$ has a unique non-degenerate zero at $(0, \zeta)$ with $\betaWD = e^{F}\beta_{k}$ over the handle satisfying the conditions of Lemma \ref{Lemma:BetaFunctionMultiple}. Over the vertical pieces we have $\betaWD = e^{s}\alpha_{\divSet}$ where $\alpha_{\divSet} = \alpha|_{T(\divSet^{-} \setminus \partial^{-}H_{k})}$
\end{proof}

\begin{defn}
A Weinstein divisor cobordism as in the above lemma associated to Weinstein surgery data $(H_{k}, \beta_{k}, \divSet^{-})$ is an \emph{ambient Weinstein surgery cobordism}.
\end{defn}

\subsection{Stabilization of contact divisors}\label{Sec:StabilizationDef}

We henceforth assume that $n \geq 2$. A \emph{cusp chart} of a critical index ambient Weinstein handle $(H_{n}, \beta_{n}) \subset \Mxi$ is a cusp chart for $\Unstable_{n}$ which is disjoint from $\partial^{+}H_{n}$. We say that $(H_{n}, \beta_{n})$ is \emph{stabilized} is there is a cusp chart for the handle within which $\Unstable_{k}$ is stabilized.

\begin{defn}
Let $\Gamma \subset \Mxi$ be a contact divisor. We say that $\Gamma$ is \emph{stabilized} if it is the result of a Weinstein anti-surgery associated to some Weinstein anti-surgery data $(H_{n}, \beta_{n}, \Gamma^{+})$ such that the handle is stabilized in a cusp chart which is disjoint from $\Gamma^{+}$.
\end{defn}

\begin{figure}[h]
\begin{overpic}[scale=.4]{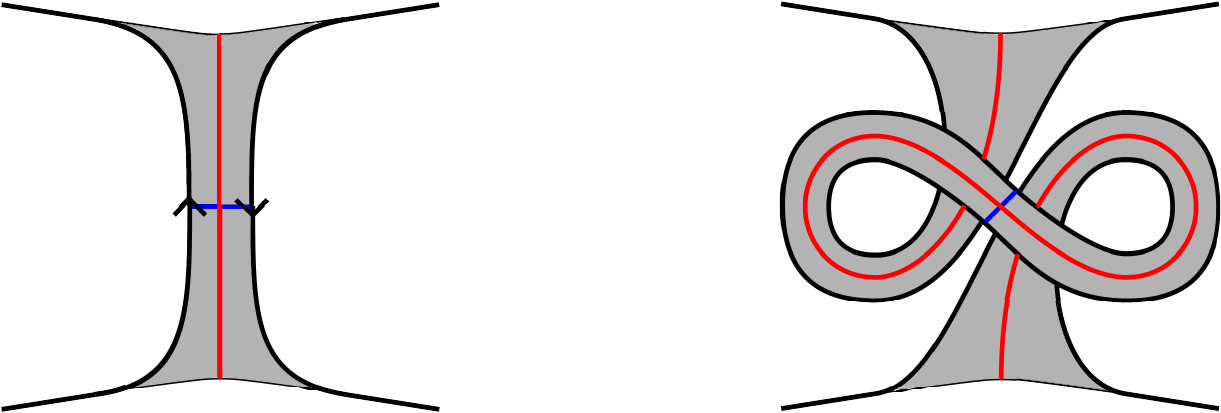}
\end{overpic}
\caption{A cartoon representation of contact stabilization. On the left, we have a contact divisor with an ambient Weinstein handle of critical index attached. On the right, the stabilized divisor is obtained by applying a Legendrian stabilization to the unstable manifold of the handle. We don't have enough dimensions to depict a stabilized ambient Weinstein handle as in Figure \ref{Fig:StabilizationFront}, so we represent the stabilized handle by twisting the original handle within the ambient space.}
\label{Fig:StableTwist}
\end{figure}

Now we describe how a given contact divisor $\Gamma \subset \Mxi$ can be \emph{stabilized along a Legendrian disk}, that is, modified to produce a new divisor which is a stabilization. We proceed in three steps.
\be
\item Find a critical index ambient Weinstein handle $(H_{n}, \beta_{n}) \subset \Mxi$ for which $(H_{n}, \beta_{n}, \Gamma)$ is Weinstein surgery data. Write $\Gamma^{+}$ for the contact divisor which is the result of the Weinstein surgery.
\item Take a cusp chart for $\Unstable_{n} \subset H_{n}$ which is disjoint from $\Gamma^{+}$. Perform a Legendrian stabilization within the cusp chart to obtain a new Legendrian disk $\Unstable_{n}'$ which agrees with $\Unstable_{n}$ outside of the cusp chart.
\item Let $(H'_{n}, \beta'_{n}) \subset \Mxi$ be an ambient Weinstein handle whose unstable manifold is $\Unstable_{n}'$ and which agrees with $(H_{n}, \beta_{n})$ outside of the cusp/stabilization chart. To ensure that $H_{n}$ and $H_{n}'$ agree outside of the cusp chart, we may shrink the $\epsilon^{\Stable}$ parameter as needed.
\ee
Finally, write $\divSet'$ for the result of Weinstein anti-surgery of $\Gamma^{+}$ along $(H'_{n}, \beta'_{n})$. See Figure \ref{Fig:StableTwist}.

\begin{defn}
In the above notation, $\Gamma'$ is the \emph{stabilization of $\divSet$ along $\Unstable_{n}$}.
\end{defn}

\subsection{Example: Standard stabilizations}\label{Sec:StandardStabilization}

According to our definition of contact stabilization, there are many ways to stabilize a given $\divSet \subset \Mxi$ which will result in a priori different contact divisors. Here we describe a \emph{standard stabilization} which can modify any $\divSet$ within a tubular neighborhood of itself.

Let $\Leg_{U}$ be a Legendrian unknot in a Darboux disk in some $\Mxi$ of $\dim=2n+1$. Then $\Leg_{U}$ admits a neighborhood of the form $N(\Leg_{U})$ as described in Equation \eqref{Eq:OneJet} which is a standard neighborhood of the Liouville hypersurface $\{ t = 0\}$.

The Liouville structure $(\disk^{\ast}\sphere^{n}, \beta_{can}=\sum p_{i}dq_{i})$ has Liouville vector field $X_{\beta_{can}} = \sum p_{i}\partial_{p_{i}}$ admitting a non-Morse Lyapunov function $\norm{p}^{2}$. The simplest ``Morse-ification'' of this Weinstein manifold yields a handle decomposition consisting of a single critical index $H_{n}$ whose stable manifold $\Stable_{n}$ is attached to the boundary of a standard disk, $(\disk^{2n}, \beta_{std})$ along a Legendrian unknot $\Leg_{U}^{n-1} = \R^{n} \cap \partial \disk^{2n}$.

We see that $\{t=0\} \times \partial\disk^{2n}$ is a standard contact unknot $\unknot$ in $N(\Leg_{U})$ and that $(H_{n}, \beta_{n}, \unknot)$ is Weinstein surgery data. So we can define $\unknot'$ to be a stabilization defined using this surgery data. Assuming Theorem \ref{Thm:MainOT}, which will be proved in \S \ref{Sec:MainProofs}, the existence of this stabilized $\unknot'$ proves Theorem \ref{Thm:UnknotNonsimple}.

For a general $\divSet \subset \Mxi$, define its \emph{standard stabilization} to be the result of a $\ind=1$ ambient Weinstein surgery connecting $\divSet$ to $\unknot'$ where the Darboux disk used in the above construction is disjoint from $\divSet$.

\subsection{Basic properties of contact stabilizations}\label{Sec:StabilizationProperties}

Now we collect some basic properties of contact stabilizations, proving most of the statements in Theorem \ref{Thm:MainOT}.

\begin{thm}
If $\divSet' \subset \Mxi$ is a stabilization of some $\Gamma \subset \Mxi$ then the divisors have the same intrinsic contact structures and are formally contact isotopic.
\end{thm}

\begin{proof}
That $\divSet$ and $\divSet'$ have the same intrinsic contact structures is clear. They are both the results of the same contact anti-surgery along the Legendrian sphere $\partial \Unstable_{n} = \partial \Unstable_{n}' \subset \divSet^{+}$ with the same framing.

Regarding contact isotopy, we follow the proof of \cite[Lemma 3.4]{CasalsEtnyre} exactly, including details for completeness of exposition. The strategy is as follows: Since the Legendrian disk $\Unstable_{n}$ and its stabilization $\Unstable_{n}'$ are formally Legendrian isotopic within the cusp chart, we can get the desired result by upgrading a formal Legendrian isotopy to an isotopy of $H_{n}$ within the cusp chart which restricts to a formal contact isotopy along $\partial^{-}H_{n}$.

First we set some notation. Let $(\phi_{T}, \Phi_{S, T})$ be a formal Legendrian isotopy with domain $\disk^{n}$ for which
\be
\item $\phi_{0}(\disk^{n})=\Unstable_{n}$ with $\Phi_{S, 0}$ constant in $S$,
\item $\phi_{1}(\disk^{n})=\Unstable_{n}'$ with $\Phi_{S, 1}$ constant in $S$, and
\item $\phi_{T}$ and $\Phi_{S, T}$ are constant outside of the cusp chart $\disk^{2n+1}$.
\ee
Let $A_{S, T}$ be a two-parameter family of automorphisms of $TM$ for which
\begin{equation*}
A_{0, T}=\Id_{TM}, \quad A_{S, T}T\phi_{T}(T\disk^{n}) = \Phi_{S, T}(T\disk^{n}).
\end{equation*}
Then we have a $T\in [0, 1]$ family of formal contact structures
\begin{equation*}
(\xi_{T}, \omega_{T}) = \left(A_{1, 1}^{-1}\xi, d\alpha(A_{1, T}^{-1}\ast, A_{1, T}^{-1}\ast)\right)
\end{equation*}
for which each $\phi_{T}\disk^{n}$ is Legendrian with respect to $(\xi_{T}, \omega_{T})$. Consideration of the two-parameter family $\left(A_{S, 1}^{-1}\xi, d\alpha(A_{S, T}^{-1}\ast, A_{S, T}^{-1}\ast)\right)$ makes it apparent that $(\xi_{T}, \omega_{T})$ is homotopic through formal contact structures to the constant family $(\xi_{T}, \omega_{T})=(\xi, d\alpha)$. We require that $A_{S, T}$ differs from $\Id_{TM}$ only within the cusp chart, so that $(\xi_{T}, \omega_{T})$ is constant outside of the cusp chart with $(\xi_{0}, \omega_{0}) = (\xi_{1}, \omega_{1})=(\xi, d\alpha)$. To obtain the desired result it suffices to construct a $T$-family of embeddings $\psi_{T}\divSet$ which are almost contact with respect the $(\xi_{T}, \omega_{T})$.

Let $g_{T}$ be a $1$-parameter family of Riemannian metrics and write $E_{T} \subset \xi_{T}$ for the $g_{T}$ orthogonal complement of $T\phi_{S}\disk^{n}$. Also write $E_{S}(\epsilon) \subset E$ for the radius $\epsilon$ disk bundle. We assume that the $\exp_{\phi_{T}}(E_{T}(\epsilon^{\Stable}))$ agrees with the ``thin part'' of $H_{n}$ when $T=0$ and agrees with the thin part of $H_{n}'$ when $T=1$. The ``thin part'' is $\{ \norm{v_{\Stable}} \leq \epsilon^{\Stable}\}$ with $\epsilon^{\Stable} \ll \epsilon^{\Unstable}$ following the notation of \S \ref{Sec:HandleAbstract} and $\exp_{\phi_{T}}$ is defined using $g_{T}$.

Then we define $H_{n, T}$ by applying $\exp_{\phi_{T}}$ to $E_{T}(\epsilon)$. For $\epsilon=\epsilon^{\Stable}$ sufficiently small, the $\omega_{T}$ will be fiber-wise symplectic on the tangent spaces of $H_{n, T}$ for all $T$, as they are symplectic along the $0$-section $\phi_{T}\Stable_{n} \subset H_{n, T}$ of the thin part. So $H_{n, 0} = H_{n}$ and $H_{n, 1} = H_{n}'$, giving the stabilized handle. So we define $\psi_{T}\divSet$ to be the negative boundary $\partial^{-}H_{n, T}$ obtained by applying $\exp_{\phi_{T}}$ to the unit sphere bundle in $E_{T}$. For the almost contact structure on $\psi_{T}\divSet$ we can take
\begin{equation*}
\xi_{\phi_{T}\Gamma} = \{ v \in \psi_{T}\divSet\ :\ \omega_{T}(v, \eta_{T})\} \subset T\Gamma
\end{equation*}
where $\eta_{T}$ is a vector field defined along the boundary of $H_{n, T}$ which points inward with $\eta_{T}$ agreeing with the Liouville vector field of the handle when $T=0, 1$. Observe that $\xi_{\phi_{T}\Gamma}$ agrees exactly with $\xi_{\Gamma}$ away from the thin part of the handle. The two form $\omega_{T}$ is symplectic on $\xi_{\phi_{T}\Gamma}$ and this subspace can be assumed to be arbitrarily close to $\xi_{T} \cap T\phi_{T}\divSet$ by making $\epsilon^{\Stable}$ very small. Hence there are $\Phi_{S, T}$ for which the $\psi_{T}\divSet$ give almost contact embeddings into $M$ equipped with the almost contact structures $(\xi_{T}, \omega_{T})$. To complete the proof concatenate the fiber-wise $\Psi_{S, T}$ homotopies in the variable $S$ with the $A_{S, T}^{-1}$, to obtain almost contact embeddings into $\Mxi$.
\end{proof}

\begin{lemma}\label{Lemma:}
Suppose that the complement $(M \setminus \divSet, \xi)$ of $\divSet$ is overtwisted. Then $\divSet$ is a stabilization.
\end{lemma}

As mentioned in the introduction, this is immediate from \cite{CP:TransverseHprinciple}. For results on low-dimensional overtwisted complements, see for example, the foundational \cite{Etnyre:OTComplement} or the more recent \cite{CEMM}.

\begin{comment}
\begin{proof}
We sketch an alternate proof using Theorem \ref{Thm:CMP}. Let $(H_{n}, \beta_{n})$ be an ambient Weinstein handle for which $(H_{n}, \beta_{n}, \divSet)$ is Weinstein surgery data, with $H_{n}$ contained in a Darboux ball $\disk^{2n+1}$ for which $M \setminus (\disk^{2n+1} \cup \divSet)$ is overtwisted. We can connect the unstable manifold $\Unstable_{n}$ of the handle to a Legendrian unknot $\Lambda_{U}$ in $M \setminus (\disk^{2n+1} \cup \divSet)$ by an isotropic arc which is disjoint from $\divSet^{+}$. Then $\Unstable_{n}$ is boundary-relative Legendrian isotopic to the Legendrian connected sum $\Lambda_{U}\#\Unstable_{n}$ through Legendrians which are disjoint from $\divSet^{+}$. Indeed such a Legendrian isotopy can be assumed to supported along a neighborhood of the isotropic arc. Since $\Lambda_{U}$ is a stabilization, then $\Unstable_{n}$ is as well.
\end{proof}
\end{comment}

\subsection{Contact anti-stabilization}\label{Sec:AntiStabilization}

Using the language of ambient handle attachments, the contact divisors of \cite{CasalsEtnyre} are obtained by applying Legendrian stabilization to the \emph{stable} manifolds of critical index Weinstein handles. Meanwhile, in our definition of contact stabilization we are performing Legendrian stabilizations along the \emph{unstable} manifolds of critical index handles. We describe more specifically how the construction of \cite{CasalsEtnyre} can be applied to general contact divisors via an \emph{anti-stabilization} procedure.

\begin{figure}[h]
\begin{overpic}[scale=.4]{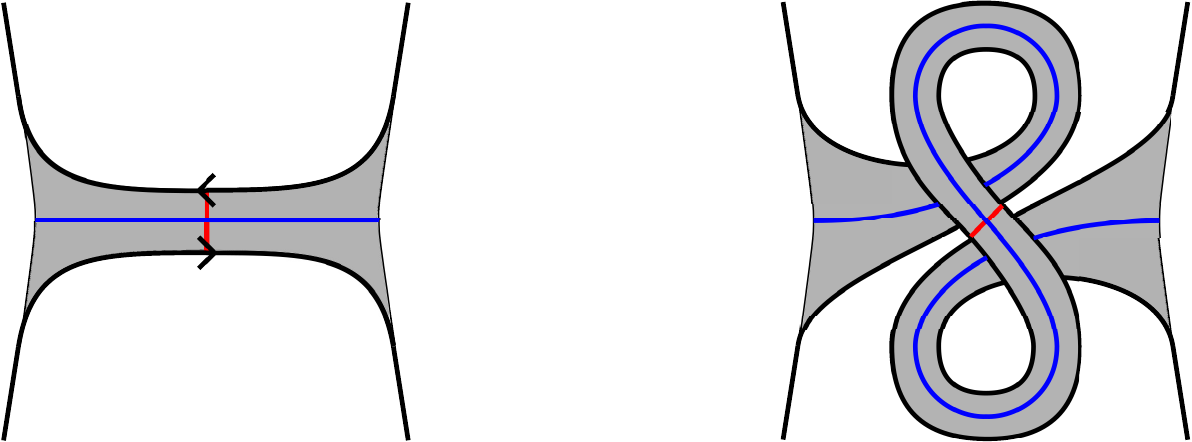}
\end{overpic}
\caption{A cartoon representation of contact anti-stabilization, following Figure \ref{Fig:StableTwist}. On the left, we have a contact divisor with an ambient Weinstein handle of critical index attached. On the right, the anti-stabilized divisor is obtained by applying a Legendrian stabilization to the stable manifold of the handle.}
\label{Fig:StableAntiTwist}
\end{figure}

Let $\divSet \subset \Mxi$ be a contact divisor.
\be
\item Find a critical index Weinstein handle $(H_{n}, \beta_{n}) \subset \Mxi$ for which $(H_{n}, \beta_{n}, \divSet)$ is Weinstein anti-surgery data. Write $\divSet^{-}$ for the contact divisor which is the result of the Weinstein anti-surgery.
\item Take a cusp chart for $\Stable_{n} \subset H_{n}$ which is disjoint from $\divSet^{-}$. Perform a Legendrian stabilization within the cusp chart to obtain a new Legendrian disk $\Stable_{n}'$ which agrees with $\Stable_{n}$ outside of the cusp chart.
\item Let $(H'_{n}, \beta'_{n}) \subset \Mxi$ be an ambient Weinstein handle whose stable manifold if $\Stable'_{n}$ and which agrees with $(H_{n}, \beta_{n})$ outside of the cusp/stabilization chart. To ensure that $H_{n}$ and $H'_{n}$ agree outside of the chart, we may shrink the $\epsilon^{\Unstable}$ parameter.
\ee
Finally, write $\divSet'$ for the result of the Weinstein surgery of $\divSet^{-}$ along $(H_{n}', \beta_{n}')$. See Figure \ref{Fig:StableAntiTwist}.

\begin{defn}
In the above notation $\divSet'$ is the \emph{anti-stabilization} of $\divSet$ along $\Stable_{n}$.
\end{defn}

The following result for anti-stabilization follows from the same proof techniques used for the analogous results for stabilization.

\begin{thm}
If $\divSet' \subset \Mxi$ is an anti-stabilization of some $\divSet \subset \Mxi$, then the divisors have the same implicit contact structures and are formally contact isotopic. If $\divSet$ is any contact divisor for which $(M \setminus \divSet, \xi)$ is overtwisted, then $\divSet$ is an anti-stabilization.
\end{thm}

For the main examples of \cite{CasalsEtnyre}, one starts with $\divSet$ equal to a contact push-off of the standard Legendrian unknot $\Leg_{U}$, takes $(H_{n}, \beta_{n})$ to be the index $n$ Weinstein handle of the Liouville hypersurface $(\disk^{\ast}\sphere^{n}, \beta_{can})$ which it bounds, and then performs an anti-stabilization along its stable manifold. The resulting $\divSet'$ is the contact push-off of a Legendrian stabilization of $\Leg_{U}$. So in contrast with Theorem \ref{Thm:HypersurfaceNonSimple} concerning stabilized contact divisors, anti-stabilized divisors may bound Liouville hypersurfaces. Meanwhile the following will be clear from the proof of Theorem \ref{Thm:OTDivisor} in \S \ref{Sec:OTDivisor}.

\begin{thm}
Suppose that $\divSet \subset \Mxi$ is a contact divisor for which $\dim M \geq 7$ and $\MxiDivSet$ is overtwisted. Then $\divSet$ is an anti-stabilization.
\end{thm}

\subsection{High-dimensional stabilization applied to transverse links}\label{Sec:HighDimLowDim}

Now we apply our definition of high-dimensional contact stabilization to transverse links. The content of this subsection will not be used elsewhere in the article and is included to demonstrate that our high-dimensional stabilization does not agree with the usual notion of stabilization for transverse links. The procedure follows Figure \ref{Fig:LowDimStabilization}, building on the description of $\ind=1$ ambient Weinstein handle attachment for transverse links in \S \ref{Sec:LowDimHandles}.

Suppose we have a Legendrian arc $\Stable$ with boundary on a transverse link $\divSet$ as shown in the first subfigure. We take this to be our stable manifold for an $\ind=1$ ambient Weinstein handle. After the surgery is applied, we obtain a transverse link as described in the second subfigure attached to which is the unstable manifold $\Unstable$ of the handle. In the third subfigure a Legendrian stabilization is performed to obtain a $\Unstable'$. In the fourth subfigure a Legendrian isotopy is applied to the union of the transverse link and $\Unstable'$. The isotopy shrinks $\Unstable'$, replacing the portion of $\Unstable'$ with the portion containing the zigzag with the \emph{ribbon} of the arc, following the language of \cite{Avdek:ContactSurgery}. The same article provides a detailed description of the resulting transverse link in the front projection. In the fourth subfigure, we undo the ambient Weinstein handle attachment by adding a negative crossing to the diagram, following the second row of Figure \ref{Fig:LowDimHandles}. In the last subfigure, a transverse isotopy has been applied to simplify the front diagram.

\begin{figure}[h]
\begin{overpic}[scale=.35]{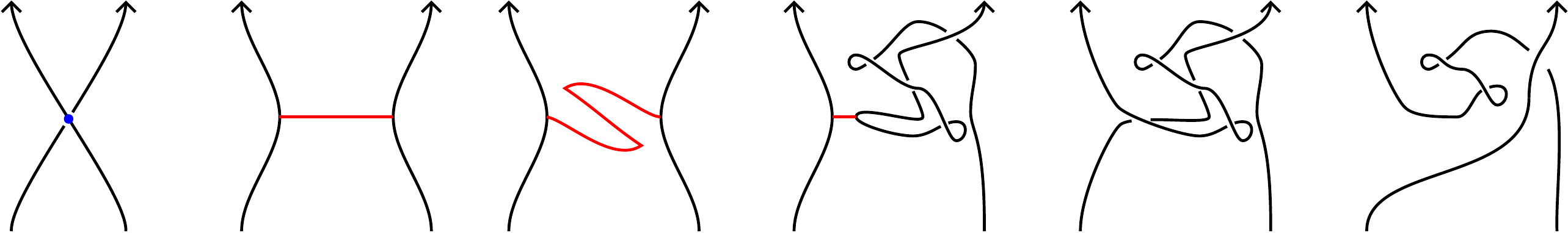}
\end{overpic}
\caption{High-dimensional stabilization applied to a transverse link in five steps.}
\label{Fig:LowDimStabilization}
\end{figure}

In the end we see that the following changes have occurred:
\be
\item A crossing change has been applied so that what was previously the under-crossing arc (south-west to north-east) is now the over-crossing arc.
\item What was originally the over-crossing arc (south-east to north-west) has now been transversely stabilized in the usual sense \cite{Etnyre:KnotNotes}.
\ee
Now let's suppose for simplicity that we are working in $\Rthree$ so that the picture extends to a full front projection. If the two strands lie on distinct components of $\divSet$, the first observation tells us that the stabilization modifies the linking numbers of the components so the topological type of the link has been modified. Otherwise we can assume that $\divSet$ is connected and see that the operation has stabilized $\divSet$ while maintaining the self-linking number (which can be computed as the writhe of the front).

If the Legendrian arc $\Unstable$ is oriented left to right, then we have applied a negative Legendrian stabilization to $\Unstable$ in the third subfigure of Figure \ref{Fig:LowDimStabilization}. If we instead perform a positive Legendrian stabilization, then the result of the high-dimensional contact stabilization will also produce a crossing change, but leave the south-east to north-west strand intact while stabilizing the south-west to north-east strand. If we perform both a positive and negative Legendrian stabilization to $\Unstable$, then we will see a crossing change and both stands of $\divSet$ transversely stabilized. The details of the alternate stabilization operations (as well as the anti-stabilization) can be worked out diagrammatically as in Figure \ref{Fig:LowDimStabilization}, using the previously-mentioned techniques of \cite{Avdek:ContactSurgery}.

\section{Normal bundles of divisors and contact fiber sums}\label{Sec:NormalAndFiberSum}

Here we collect some results on framings of contact divisors, how these framings interact with ambient surgeries, and contact fiber sums. These concepts will be synthesized in \S \ref{Sec:FiberSum}.

\subsection{Framings of contact divisors}\label{Sec:Framing}

The first Chern class of the normal bundle $\eta$ of a contact divisor $\divSet \subset \Mxi$ can be computed using the adjunction and self-intersection formulas
\begin{equation*}
c_{1}(\eta) = c_{1}(\xi) - c_{1}(\xiDivSet) = i^{\ast}\PD_{M}([\divSet]) \in H^{2}(\divSet, \Z).
\end{equation*}
Here $\PD_{M}: H_{2n-1}(M) \rightarrow H^{2}(M)$ and $i^{\ast}:H^{2}(M) \rightarrow H^{2}(\divSet)$ is the pullback of the inclusion. The first equality (adjunction) is derived from the more general adjunction formula $c(\xi|_{\divSet}) = c(\eta)c(\xiDivSet)$ for total Chern classes. The second (self-intersection) is simply the fact that $c_{1}(\eta)$ is the Euler class.

A \emph{framing} of a contact divisor is a nowhere vanishing section of its normal bundle. The normal bundle admits a framing iff $c_{1}(\eta)$ vanishes. When framings exist, their homotopy classes are an affine space over $H^{1}(\divSet, \Z) \simeq [\divSet, \Circle]$. It follows that divisors which are spheres -- such as the unknot -- admit unique framings when $\dim \divSet \geq 3$.

We designate a special class of framings. If $\divSet$ is a link bounding an oriented, embedded hypersurface $S$, then the intersection of $S$ with the boundary of a small neighborhood of $\divSet$ in $M$ gives a framing of $\divSet$. We say that such a framing is a \emph{Seifert framing} and denote it by $\framing_{S}$. This coincides with the classical notion of Seifert framings of knots in $\sphere^{3}$.

\subsection{Framings and ambient surgeries}

Provided a framing $\framing$ of a divisor $\divSet \subset \Mxi$, the relevant Weinstein neighborhood theorem states that the divisor admits a neighborhood of the form
\begin{equation*}
N_{\Gamma} = \disk \times \Gamma, \quad \xi=\ker \alpha, \quad \alpha = r^{2}d\theta + e^{r^{2}}\alpha_{\Gamma}, \quad \framing = \partial_{x}
\end{equation*}
where $\alpha_{\Gamma}$ is a contact form for $\MxiDivSet$ and $\disk$ a $\dim=2$ disk of sufficiently small radius. In the above we use both polar $(r, \theta)$ and Cartesian $(x, y)$ coordinates on $\disk$.

Within such neighborhood, we have a Liouville embedding of the symplectization $([C_{0}, C_{1}]\times \divSet, e^{s}\alpha_{\Gamma})$ of $\MxiDivSet$ into $\Mxi$ defined locally by
\begin{equation*}
(s, \gamma) \mapsto (r=\sqrt{s}, \theta=0, \gamma), \quad \gamma \in \Gamma
\end{equation*}
for $0 < C_{0} < C_{1}$ small. We'll say that such a Weinstein hypersurfaces is a \emph{ribbon} of $(\divSet, \framing)$, whose convex and concave boundaries are both contact isotopic to $\divSet \subset \Mxi$. To extend this to an embedding of the form $[-\epsilon, \epsilon]_{t}$ times the ribbon, flow in the Reeb direction for time $t$. Ribbons will be denoted $\ribbon_{\divSet, \framing}$ and we will allow ourselves to identify either its concave or convex boundary with $\divSet$, depending on the context at hand.

\begin{lemma}\label{Lemma:HandleLiouvilleExtension}
Let $(\divSet, \framing)$ be a framed divisor in a $\dim=2n+1$ contact manifold $\Mxi$ for which $n \neq 2$ or $\framing$ is a Seifert framing. Let $(H_{n}, \beta_{n})$ be an ambient Weinstein handle for which $(H_{n}, \beta_{n}, \divSet)$ is Weinstein surgery data. Then there is a ribbon $\ribbon$ for some $(\divSet, \framing')$ for which
\be
\item $\framing'$ homotopic to $\framing$,
\item $\divSet$ is the convex boundary of $\ribbon$, and
\item the union of $\ribbon$ with $H_{n}$ is a Weinstein hypersurface.
\ee
Likewise if $(H_{n}, \beta_{n}, \divSet)$ is Weinstein anti-surgery data, then there is a ribbon $\ribbon$ for some $(\divSet, \framing')$ for which
\be
\item $\framing'$ is homotopic to $\framing$,
\item $\divSet$ is the concave boundary of $\ribbon$, and
\item the union of $\ribbon$ with $H_{n}$ is a Weinstein hypersurface.
\ee
\end{lemma}

\begin{proof}
We'll only cover the Weinstein surgery case as the Weinstein anti-surgery case is identical. We want to ensure that the union of a sufficiently thin ribbon with the handle is embedded in $\Mxi$. We first consider the associated homotopical problem.

Use the framing to identify $N_{\divSet}$ with $\disk \times \divSet$ so that $\framing = \partial_{x}$ as above. The handle deformation-retracts onto its stable manifold $\Stable_{n}$ with $\partial \Stable_{n} \simeq \sphere^{n-1} \subset \divSet$. Let $\zeta=X_{\beta_{n}}|_{\partial \Stable_{n}}$, giving an inward pointing vector field for $\Stable_{n}$ along its boundary.  Using the trivialization of the normal bundle $\eta$ determined by $\framing$, $\zeta$ gives us a map $\sphere^{n-1} \rightarrow \Circle = \partial \disk$ which we'll also call $\zeta$ with homotopy class $[\zeta]$.

We claim that $[\zeta]$ is in the class of the constant map provided our hypotheses. When $n > 2$, this is trivially satisfied, so we assume that $n=2$ and $\framing$ is a Seifert framing. So $[\zeta]$ is completely determined by
\begin{equation*}
\deg = \deg(\zeta: \Circle \rightarrow \Circle) \in \Z.
\end{equation*}
We can compute this degree using the intersection pairing $H_{3}(M \setminus \divSet) \otimes H_{2}(M, \divSet) \rightarrow \Z$ as
\begin{equation*}
\deg = \langle \{ (\epsilon, 0)\} \times \divSet, \Stable_{n} \rangle.
\end{equation*}
This must be zero since the cycle $\{ (\epsilon, 0)\} \times \divSet$ is the boundary of a Seifert hypersurface $S$ minus its intersection with our neighborhood $S \cap N_{\divSet}$.

So we suppose that $[\zeta]$ is in the class of the constant map. Then can find a smooth $\framing'$ in the same homotopy class of framing as $\framing$ which agrees with $X_{\beta_{n}}$ along the attaching locus $\partial^{-}H_{n}$ of the handle. Take a ribbon $\ribbon$ of $(\divSet, \framing')$ whose convex end we identify with $\divSet$. The union of this ribbon with $H_{n}$ will then be smoothly embedded so long as $\ribbon$ is sufficiently thin, giving us the desired Weinstein hypersurface.
\end{proof}

\subsection{Push-offs of divisors and Liouville hypersurfaces}\label{Sec:Pushoffs}

Provided a contact divisor $\divSet$ with a framing $\framing$, we can define the \emph{$\framing$ push-off}, denoted $\divSet_{\framing}$, as follows: Use $\framing$ to identify $\divSet = \{0\} \times \divSet$ and $\framing = \partial_{x}$ inside of a trivialized $N_{\divSet}$ as above and define
\begin{equation*}
\divSet_{\framing} = \{(\epsilon, 0) \} \times \divSet \subset N_{\divSet} \subset M \setminus \divSet, \quad \epsilon > 0.
\end{equation*}
Following the usual Darboux-Moser-Weinstein arguments, the contact isotopy classes of $\divSet_{\framing} \subset M \setminus \divSet$ and $\divSet \sqcup \divSet_{\framing}$ depend only on the homotopy class $[\framing]$ of $\framing$.

We can also describe push-off operations for Liouville hypersurfaces $(\Wdivisor, \betaWD) \subset \Mxi$. Inside of the standard neighborhood $N_{\Wdivisor} = [-\epsilon, \epsilon]_{t} \times \Wdivisor$ within which $\xi = dt + \beta_{V}$ and $\Wdivisor$ is identified with $\{0\} \times \Wdivisor \subset N_{\Wdivisor}$, set
\begin{equation*}
\Wdivisor_{\uparrow} = \{\epsilon/2\} \times \Wdivisor.
\end{equation*}
Then $\Wdivisor_{\uparrow}$ is a Liouville hypersurface with the same implicit Liouville structures $\alpha|_{\Wdivisor_{\updownarrow}} = \beta$. If $\divSet^{\pm} = \partial^{\pm}\Wdivisor$ is one of the boundary components, then the $\divSet^{\pm}_{\uparrow} = \partial^{\pm}\Wdivisor_{\uparrow}$ are both $\framing_{\Wdivisor}$ push-offs of $\divSet^{\pm}$.

The following easy lemma will later be used in the proof of Theorem \ref{Thm:MainOT}.

\begin{lemma}\label{Lemma:DisjointUnknots}
Let $\unknot$ be the standard contact unknot inside of a $\Mxi$ of $\dim \geq 5$. Then for each $\framing$ push-off $\divSet_{U, \framing}$ of $\unknot$, there are disjoint Weinstein disks bounded by $\unknot$ and $\divSet_{U, \framing}$.
\end{lemma}

\begin{proof}
As previously mentioned, there is only one $\framing$ of the unknot up to homotopy, which must then agree with the framing induced by the standard Liouville disk $\Wdivisor \simeq \disk^{2n}$ which it bounds. Since the contact isotopy class of $\divSet_{U, \framing}$ depends only on the homotopy class of the framing, we can assume that it is the convex boundary of $\Wdivisor_{\uparrow}$.
\end{proof}

\subsection{Fiber sums of framed contact divisors}\label{Sec:FiberSumIntro}

Let $\MxiDivSet$ and $\Mxi$ be contact manifolds of $\dim=2n-1$ and $\dim=2n+1$ respectively. Suppose we have a pair of contact embeddings $\phi_{i}: \MxiDivSet \rightarrow \Mxi$ whose images are disjoint and have framings $\framing_{i}, i=0, 1$. To this data Geiges associates a \emph{contact fiber sum} \cite{Geiges:Branched} as a contact-topological analogue of Gompf's symplectic fiber sum \cite{Gompf:Sum}. We review a constructive definition, producing a contact manifold which we'll denote $(M_{\#}, \xi_{\#})$. For an intrinsic characterization of the contact fiber sum using confoliations, see Gironella's \cite{Fabio:Thesis}.

The contact fiber sum construction works by identifying tubular neighborhoods of the $\phi_{i}\divSet$ with $\disk \times \divSet$ along which $\xi$ is given by the contact form $\alpha = \alpha_{\divSet} + r^{2}d\theta$ and $\framing = \partial_{x}$ for a contact form $\alpha_{\divSet}$ for $\MxiDivSet$. The complements of the $\phi_{i}\divSet$ in their neighborhoods are $(0, \epsilon]_{r} \times \Circle_{\theta} \times \divSet$ in polar coordinates. These complements can be identified with subsets of $[-\epsilon, \epsilon]_{p} \times \Circle_{q} \times \divSet$ via the embeddings
\begin{equation*}
(r, \theta, z) \mapsto (p=(-1)^{i}r, q=(-1)^{i}\theta, z).
\end{equation*}
We can then equip
\begin{equation*}
M_{\#} = (M \setminus (\phi_{0}\divSet \sqcup \phi_{1}\divSet)) \cup \left( [-\epsilon, \epsilon]_{p} \times \Circle_{q} \times \divSet\right)
\end{equation*}
with a contact form $\alpha_{\#}$ which agrees with $\alpha$ outside of the overlapping region $[-\epsilon, \epsilon]_{p} \times \Circle_{q} \times \divSet$. Within the overlapping, set
\begin{equation*}
\alpha_{\#} = \rho(p)dq + \alpha_{\divSet}
\end{equation*}
with $\rho(p)$ a strictly increasing function of the form
\begin{equation*}
\rho(p) = \begin{cases}
-p^{2} & p < -\epsilon/2,\\
p & p \in [-\epsilon/4, \epsilon/4],\\
p^{2} & p > \epsilon/2.
\end{cases}
\end{equation*}
The induces contact structure $\xi_{\#}$ is independent of $\rho$ by standard Darboux-Moser-Weinstein arguments, cf. \cite[\S 3]{MS:SymplecticIntro}.

For an interesting application of the contact fiber sum, see \cite[Theorem I.5.1]{Klaus:Habilitation} (due to Presas): Taking fiber sums along divisors whose implicit contact structures are overtwisted produces closed contact manifolds which contain plastikstufes \cite{Klaus:Plastik} and are therefore overtwisted in the sense of \cite{BEM:Overtwisted} by \cite{CMP:OT}. This result is the inspiration for Theorem \ref{Thm:FiberSumOT}, below. Fiber sums of supporting open books along their bindings (with framings given by pages) are studied in \cite{Avdek:Liouville} and will be used in \S \ref{Sec:MainProofs}.

\section{Fiber sum cobordisms}\label{Sec:FiberSum}

Let $(\Wdivisor, \betaWD)$ be a $2n$-dimensional Liouville cobordism and consider disjoint Liouville embeddings $\phi_{i}(\Wdivisor) \subset \Mxi, i=0,1$ into a contact manifold $\Mxi$ of dimension $2n+1$. Then
\begin{equation*}
\divSet^{\pm}_{i}= \phi_{i}\left(\partial^{\pm}\Wdivisor\right)
\end{equation*}
consists of four pair-wise disjoint contact divisors of $\Mxi$. The maps $\phi_{i}$ restricted to $\partial^{-}\Wdivisor$ (resp. $\partial^{+}\Wdivisor$) provides us an identification of the $\divSet^{-}_{i}$ (resp. $\divSet^{+}_{i}$). Each is framed in $\Mxi$ as a boundary component of the $\phi_{i}\Wdivisor$. Therefore we can use the framings to define contact fiber sums $(M^{\pm}_{\#}, \xi^{\pm}_{\#})$ of $\Mxi$ along the pairs of framed contact divisors $(\divSet^{\pm}_{0},\framing_{\phi_{0}\Wdivisor}), (\divSet^{\pm}_{1}, \framing_{\phi_{1}\Wdivisor})$. When $\partial^{-}\Wdivisor = \emptyset = \divSet^{-}_{i}$, we declare $(M^{-}, \xi^{-}) = \Mxi$.

\begin{thm}\label{Thm:FiberSumCobordism}
In the above notation, there is a $\dim=2n+2$ Liouville cobordism $(W_{\#}, \beta_{\#})$ with convex boundary $(M^{+}_{\#}, \xi^{+}_{\#})$ and concave boundary $(M^{-}_{\#}, \xi^{-}_{\#})$. If $(\Wdivisor, \betaWD)$ is Weinstein with Lyapunov function $F_{\Wdivisor}$, then the cobordism is Weinstein for a Lyapunov function $F_{\#}$ admitting a handle decomposition with two $\ind = m+1$ handles for each $\ind = m$ handle in $\Wdivisor$. It follows that if $(\Wdivisor, \betaWD, F_{\Wdivisor})$ is a subcritical Weinstein cobordism, then $(W_{\#}, \beta_{\#}, F_{\#})$ is as well.
\end{thm}

To build the cobordism we will apply a symplectic fiber sum to lifts $\widetilde{\phi_{i}(V)} \subset [-C, C] \times M$ of the Liouville embeddings $\phi_{i}(V)$ in the style of \cite{Gompf:Sum}. The restriction of the fiber sum to the boundary will be a contact fiber sums described in \S \ref{Sec:FiberSumIntro}. In the Weinstein case, it is difficult to concisely describe the Weinstein handle decomposition of $W_{\#}$ without getting into the details of the proof. See \S \ref{Sec:UnstableAnalysis}. However, the following lemma is easy to state.

\begin{lemma}\label{Lemma:FiberSumOT}
In the above notation, suppose that $(\Wdivisor, \betaWD)$ has only one critical point of $\ind=n$ with unstable manifold $\Unstable_{n}$. Suppose further that one of the $\phi_{i}(\Unstable_{n})$ is stabilized in the complement of the $\divSet^{+}_{i}$. Then $(M^{-}_{\#}, \xi^{-}_{\#})$ is the result of a contact anti-surgery on a stabilized Legendrian sphere, and so is overtwisted by Theorem \ref{Thm:CMP}.
\end{lemma}

Using this lemma, we prove the following theorem at the end of this section.

\begin{thm}\label{Thm:FiberSumOT}
Suppose that $\phi_{i}, i=0,1$ are contact embeddings of a $\dim=2n-1$ contact manifolds $\MxiDivSet$ into a $\dim=2n+1 \geq 5$ contact manifolds satisfying the following conditions:
\be
\item The $\phi_{i}\divSet$ are equipped with framings $\framing_{i}$ so that their fiber sum $(M_{\#}, \xi_{\#})$ is defined.
\item $n \neq 2$ or $\framing_{0}$ is a Seifert framing.
\item $\phi_{0}\divSet$ is stabilized in the complement of $\phi_{1}\divSet_{i}$.
\ee
Then the fiber sum $(M_{\#}, \xi_{\#})$ of the $(\phi_{i}\divSet, \framing_{i})$ is overtwisted.
\end{thm}

These results follow from analysis of unstable manifolds of the $\ind = n+1$ critical points of $W_{\#}$. Similar results are available for stable manifolds, which can be used to study anti-stabilizations. While we have no new applications of anti-stabilization in this paper -- many applications are provided by \cite{CasalsEtnyre} -- we include the following for completeness.

\begin{lemma}\label{Lemma:FlexibleFiberSum}
In the above notation, suppose that $(\Wdivisor, \betaWD)$ is Weinstein and that for each of its critical-index stable manifolds $\Stable(\zeta, X_{\betaWD})$ at least one of the Legendrian disks $\phi_{i}(\Stable(\zeta, X_{\betaWD}))$ is Legendrian stabilized in the complement of $\divSet_{0}^{-} \sqcup \divSet_{1}^{-}$. Then the Weinstein cobordism $(W_{\#}, \beta_{\#})$ is flexible.
\end{lemma}

\subsection{Normalized framings}

The first step in the contact and symplectic fiber sum construction is the description of Weinstein neighborhood determined by a normal bundle of the submanifold. We describe our neighborhoods using the following definition and lemma.

\begin{defn}
Let $\widetilde{\Wdivisor}$ be a $\codim=2$ symplectic submanifold of some symplectic manifold $(W, \omega)$. A \emph{normalized framing} of $\widetilde{\Wdivisor}$ is a pair of vector fields $\eta_{x}, \eta_{y} \in \Sec(TW|_{\widetilde{\Wdivisor}})$ such that at each point in $\widetilde{\Wdivisor}$
\be
\item $\Span(\eta_{x}, \eta_{y}) = T\widetilde{\Wdivisor}^{\omega} = \{ v \in TW\ : \ \omega(v, \ast)|_{T\widetilde{\Wdivisor}} = 0\}$ and
\item $\omega(\eta_{x}, \eta_{y}) = 1$.
\ee
\end{defn}

\begin{lemma}\label{Lemma:LiouvilleTaylor}
Let $\WdivisorLift \subset W$ be a Liouville submanifold of a $(W, \betaW)$ with normalized framing $\eta_{x}, \eta_{y}$ and write $\betaWDLift = \betaW|_{T\WdivisorLift}$. Then within any tubular neighborhood $\disk_{x, y} \times \WdivisorLift$ of $\WdivisorLift$ for which $\partial_{x} = \eta_{x}$ and $\partial_{y} = \eta_{y}$ along $\{ x = y = 0\} = \WdivisorLift$, we have a fiber-wise Taylor expansion of $\betaW$ as
\begin{equation*}
\begin{gathered}
\betaW = \betaWDLift - xdA - ydB + (A + A_{x}x + A_{y}y)dx + (B + B_{x}x + B_{y}y)dy + \beta_{W, \hot},\\
A = \betaW(\eta_{x}), \quad B = \betaW(\eta_{y}), \quad B_{x} - A_{y} = 1,\\
d\betaW = d\betaWDLift + dx \wedge dy + \omega_{\hot},\\
X_{\betaW} = X_{\betaWDLift} + B\partial_{x} - A\partial_{y} + X_{\betaW, \hot},\\
\beta_{W, \hot} \in \bigO(r^{2}), \quad \omega_{\hot} \in \bigO(r), \quad X_{\betaW, \hot} \in \bigO(r),
\end{gathered}
\end{equation*}
where $A_{x}, A_{y}, B_{x}, B_{y} \in \Cinfty(\Wdivisor)$ and $r = \norm{(x, y)}$. If $F_{W}$ is a Lyapunov function for $(W, \betaW)$ restricting to a Lyapunov function for $(\WdivisorLift, \betaWDLift)$, then we have an expansion of $F_{W}$ as
\begin{equation*}
F_{W} = F_{\WdivisorLift} + xdF_{W}(\eta_{x}) + ydF_{W}(\eta_{y}) + F_{W, \hot}, \quad F_{W, \hot} = \bigO(r^{2})
\end{equation*}
\end{lemma}

\begin{proof}
That $d\betaW$ is as described is implicit in the definition of the normalized framing. Start with an expansion of the form
\begin{equation*}
\betaW = \betaWD + \beta_{\WdivisorLift, x}x + \beta_{\WdivisorLift, y}y + (A + A_{x}x + A_{y}y)dx + (B + B_{x}x + B_{y}y)dy + \beta_{W, \hot}
\end{equation*}
with $\beta_{\WdivisorLift, x}, \beta_{\WdivisorLift, y} \in \Omega^{1}(\WdivisorLift)$. The $A$ and $B$ are as defined in the statement of the lemma and we compute
\begin{equation*}
d\betaW = d\betaWDLift + (\beta_{\Wdivisor, x} + dA)\wedge dx + (\beta_{\Wdivisor, y} + dB)\wedge dy + (B_{x} - A_{y})dx\wedge dy + \omega_{\hot}
\end{equation*}
to obtain $\beta_{\WdivisorLift, x} = -dA$, $\beta_{\WdivisorLift, y} = -dB$, and $B_{x} - A_{y} = 1$.
\end{proof}

\subsection{Choice of coordinate system}

In this section we describe neighborhoods of lifts of Liouville hypersurfaces which provide normalized framings. To start, we work in a standard neighborhood $[-\epsilon, \epsilon]_{t} \times \Wdivisor$ of a Liouville hypersurface in a contact manifold $\Mxi$ along which we assume that our contact form is $\alpha = dt + \betaWD$ with $\betaWD \in \Omega^{1}(\Wdivisor)$. We choose a lifting function $F_{\Wdivisor}$ with $F_{\Wdivisor}|_{\partial^{\pm}\Wdivisor} = \pm \epsilon$, assuming that $F_{\Wdivisor}$ is a Lyapunov function for $\betaWD$ when it is Weinstein.

Then we apply the lift to obtain
\begin{equation*}
\WdivisorLift = \{ s = F_{\Wdivisor}, t=0\} \subset \disk_{s, t} \times \Wdivisor \subset [-\epsilon, \epsilon] \times M, \quad \betaW = e^{s}(dt + \betaWD).
\end{equation*}
Then $\betaWDLift = \betaW|_{\WdivisorLift} = e^{F_{\Wdivisor}}\betaWD$ and our Lyapunov function is $F_{W} = s$. Applying the map
\begin{equation*}
(s, t, z) \mapsto (s + F_{\Wdivisor}, t, z)
\end{equation*}
to our neighborhood of $\Wdivisor$ in the symplectization sends $\{s = t = 0\} \subset \disk_{s, t} \times \Wdivisor$ to $\Wdivisor \subset \disk_{s, t} \times \Wdivisor$ so that we may write
\begin{equation*}
\begin{gathered}
\betaW = e^{s+F_{\Wdivisor}}(dt + \betaWD), \quad d\betaW = e^{s+F_{\Wdivisor}}(d(s+F_{\Wdivisor})\wedge (dt + \betaWD) + d\betaWD),\\
X_{\betaW} = \partial_{s}, \quad F_{W} = s + F_{\Wdivisor}
\end{gathered}
\end{equation*}
in our transformed coordinate system.

We'll find a computationally convenient normalized framing of $\WdivisorLift$ in the above coordinate system. For the following, $X_{F_{\Wdivisor}}$ is the Hamiltonian vector field computed with respect to $d\betaWD$. Use 
\begin{equation*}
\betaWD(X_{F_{\Wdivisor}}) = d\betaWD(X_{\betaWD}, X_{F_{\Wdivisor}}) = df(X_{\betaWD})
\end{equation*}
to calculate that for $v \in T\Wdivisor$,
\begin{equation*}
\begin{gathered}
e^{-s-F_{\Wdivisor}}d\betaW(\ast, v) = \betaWD(v)d(s+F_{\Wdivisor})-df(v)(dt + \betaWD) + d\betaWD(\ast, v),\\
e^{-s-F_{\Wdivisor}}d\betaW(\partial_{s} - X_{\betaWD}, v) = e^{-s-F_{\Wdivisor}}d\betaW(\partial_{t} - X_{F_{\Wdivisor}}, v) = 0,\\
\begin{aligned}
e^{-s-f}d\betaW(\partial_{s} - X_{\betaWD}, \partial_{t} - X_{F_{\Wdivisor}}) &= (1-df(X_{\beta_{0}}))(1 - \beta_{0}(X_{\betaWD})) + \betaWD(X_{F_{\Wdivisor}})\\
&= 1 - \betaWD(X_{F_{\Wdivisor}}) + \betaWD(X_{F_{\Wdivisor}})^{2}.
\end{aligned}
\end{gathered}
\end{equation*}
Therefore we obtain a normalized framing $\eta_{x}, \eta_{y}$ defined by
\begin{equation}\label{Eq:PhiTaylor}
\begin{gathered}
\eta_{x} = \partial_{s} - X_{\betaWD}, \quad \eta_{y} = e^{-s - f}(1 - \betaWD(X_{f}) + \betaWD(X_{f})^{2})^{-1}(\partial_{t} - X_{f})\\
A = \betaW(\eta_{x}) = 0, \quad B = \betaW(\eta_{y}) = (1 - \betaWD(X_{F_{\Wdivisor}}))\left(1 - \betaWD(X_{F_{\Wdivisor}}) + \betaWD(X_{F_{\Wdivisor}})^{2}\right)^{-1},\\
X_{\betaW} = \eta_{x} + X_{\betaWD}, \quad dF(\eta_{x}) = 1 - \betaWD(X_{F_{\Wdivisor}}), \quad dF(\eta_{y}) = 0.
\end{gathered}
\end{equation}
Noting that in the Weinstein case $\tau F_{\Wdivisor}$ is Lyapunov for $\betaWD$  for any constant $\tau > 0$ we may assume, possibly after rescaling $F_{\Wdivisor}$, that
\begin{equation}\label{Eq:Bbound}
B, 1 - \betaWD(X_{F_{\Wdivisor}}) \geq C_{0} > 0
\end{equation}
point-wise for some constant $C_{0}$while maintaining the Lyapunov property.

To get extra control over our Liouville vector field, take a neighborhood $\disk_{x, y} \times \Wdivisor$ of $\Wdivisor \subset \disk_{s, t} \times M$ to be the image of the map $\Psi$ defined
\begin{equation*}
\Psi: \disk_{x, y} \times \Wdivisor \rightarrow [-\epsilon, \epsilon]_{s} \times M, \quad \Psi(x, y, z) = \Flow^{y}_{\eta_{y}}\circ \Flow^{x}_{\eta_{x}}(z)
\end{equation*}
where $z \in \Wdivisor$ and $\eta_{x}, \eta_{y}$ are as defined in Equation \eqref{Eq:PhiTaylor}. We'll call this the \emph{$\Psi$ coordinate system}.

\begin{lemma}\label{Lemma:PsiTaylor}
In the $\Psi$ coordinate system the Taylor expansion of Lemma \ref{Lemma:LiouvilleTaylor} takes the form
\begin{equation*}
\begin{gathered}
\betaW = \betaWDLift - ydB + A_{y}dx + (B + B_{x}x + B_{y}y)dy + \beta_{W, \hot},\\
B = \betaW(\eta_{y}), \quad B_{x} - A_{y} = 1,\\
d\betaW = d\betaWDLift + dx \wedge dy + \omega_{\hot},\\
X_{\betaW} = X_{\betaWDLift} + B\partial_{x} + X_{\betaW, \hot},\\
\beta_{W, \hot} \in \bigO(r^{2}), \quad \omega_{\hot} \in \bigO(r), \quad X_{\betaW, \hot} \in \bigO(r),\\
F_{W} = F_{\Wdivisor} + F_{x}x + \bigO(r^{2}), \quad F_{x} = 1 - \betaWD(X_{F_{V}}),
\end{gathered}
\end{equation*}
where $B$ is computed as in Equation \eqref{Eq:PhiTaylor}. Moreover, given a zero $\zeta \in \Wdivisor$ of $\betaWD$ we have equivalences
\begin{equation*}
\begin{aligned}
\Stable(\Stable(\zeta, X_{\betaWD}), X_{\betaW}) &:= \{ (x, y, z)\ :\ \exists T \geq 0, \ \Flow^{T}_{X_{\betaW}} \in \Stable(\zeta, X_{\betaWD})\} \\
&= \{ (x, 0, z)\ :\ x < 0, \ z \in \Stable(\zeta, X_{\betaWD})\},\\
\Unstable(\Unstable(\zeta, X_{\betaWD}), X_{\betaW}) &:= \{ (x, y, z)\ :\ \exists  T \leq 0, \ \Flow^{T}_{X_{\betaW}} \in \Unstable(\zeta, X_{\betaWD})\} \\
&= \{ (x, 0, z)\ :\ x > 0, \ z \in \Unstable(\zeta, X_{\betaWD})\}.
\end{aligned}
\end{equation*}
\end{lemma}

The last statement will later help us to understand the Weinstein structure on our cobordism. For the proof, we need to verify this last statement and see that $A_{x} = 0$. Both follow from restricting attention to $\Psi|_{y=0}$ and noting that in the $\disk_{s, t} \times \Wdivisor$ coordinate system
\begin{equation*}
\begin{aligned}
\Stable(\Stable(\zeta, X_{\beta_{0}}), X_{\betaW}) &= \{ (s, 0, z)\ :\ s \leq 0,\ z \in \Stable(\zeta, X_{\beta_{0}})\},\\
\Unstable(\Unstable(\zeta, X_{\beta_{0}}), X_{\betaW}) &= \{ (s, 0, z)\ :\ s \leq 0,\ z \in \Unstable(\zeta, X_{\beta_{0}})\}.
\end{aligned}
\end{equation*}

\subsection{Modifications of $\betaW$ in a blown-up neighborhood}

Continuing the analysis of the preceding subsection, we now deform $\betaW$ in the complement of $\Wdivisor$ inside of the $\Psi$ coordinate system. Consider $\betaW$ as in the above lemma and cutoff functions $h^{+}_{\epsilon}(r), h^{-}_{\epsilon}(r)$ to be specified. Define $\betaW^{\epsilon, \delta}$ as
\begin{equation}\label{Eq:FiberSumBetaDef}
\begin{aligned}
\betaW^{\epsilon, \delta} &= \betaWDLift - ydB + Bdy + \delta h^{-}_{\epsilon}d\theta + h^{+}_{\epsilon} \left(A_{y}ydx + (B_{x}x + B_{y}y)dy + \beta_{W, \hot}\right),\\
d\betaW^{\epsilon, \delta} &= d\betaWDLift + \left( \frac{\delta}{r} \frac{\partial h^{-}_{\epsilon}}{\partial r} + h^{+}_{\epsilon}\right)dx \wedge dy\\
&+ h^{+}_{\epsilon}d\beta_{W, \hot} + \frac{\partial h^{+}}{\partial r}dr \wedge \left(A_{y}ydx + (B_{x}x + B_{y}y)dy + \beta_{W, \hot}\right).
\end{aligned}
\end{equation}
Here we have used $dh^{\pm} = \frac{\partial h^{\pm}}{\partial r}dr = r^{-1}\frac{\partial h^{\pm}}{\partial r}(xdx + ydy)$.

We choose our cutoff functions to be of the form
\begin{equation*}
h^{\pm}_{\epsilon} = h^{\pm}_{1}\left(\frac{r}{\epsilon}\right)
\end{equation*}
where $h^{\pm}_{1}$ are smoothings of the piece-wise linear functions
\begin{equation*}
h^{-, \square}_{1}(r) = \begin{cases}
r & r < 1,\\
2-r & r \in[1,2],\\
0 & r > r
\end{cases}\quad
h^{+, \square}_{1}(r) = \begin{cases}
0 & r < 1,\\
r - 1 & r \in [1, 2],\\
1 & r > 2.
\end{cases}
\end{equation*}
We require that the smoothings $h^{\pm}_{1}$ of the $h^{\pm, \square}_{1}$ are such that
\be
\item $h^{-}(r) = r$ for $r < 1$ and $h^{-}(r) = 0$ for $r > 2$ and
\item $h^{+}$ is an increasing function with $h^{+}(r) = 0$ for $r < 1$ and $h^{+}(r) =1$ for $r > 2$.
\ee

Then $\betaW^{\epsilon, \delta}$ is defined over the compactification
\begin{equation*}
\overline{\disk^{\ast}_{x, y}} \times \Wdivisor = [0, 2\epsilon)_{r} \times \Circle_{\theta} \times \Wdivisor
\end{equation*}
of $(\disk_{x, y} \times \Wdivisor) \setminus \Wdivisor$. Since $\frac{\partial h^{+}}{\partial r}$ is bounded by a constant and $\norm{dr} = 1$ (with respect to the metric $dx^{\otimes 2} + dy^{\otimes 2}$) we can collect the terms in the last line of Equation \eqref{Eq:FiberSumBetaDef} as a higher-order-term to obtain
\begin{equation*}
d\betaW^{\epsilon, \delta} = d\betaWDLift + \left( \frac{\delta}{r} \frac{\partial h^{-}_{\epsilon}}{\partial r} + h^{+}_{\epsilon}\right)dx \wedge dy + \omega^{\epsilon, \delta}_{\hot}, \quad \omega^{\epsilon, \delta}_{\hot} \in \bigO(r).
\end{equation*}
Looking at the graph of the coefficient function for $dx \wedge dy$ it follows that for each $\epsilon > 0$, $d\betaW^{\epsilon, \delta}$ is symplectic for $\epsilon > 0$ sufficiently small.

Now we describe the Liouville field of $\beta^{\epsilon, \delta}_{\hot}$ along $\{r < \epsilon\}$. Along this subset we have precise expressions,
\begin{equation}\label{Eq:BetaModifiedNearZero}
\begin{aligned}
\betaW^{\epsilon, \delta} &= \betaWDLift  - ydB + Bdy + \delta rd\theta\\
&= \betaWDLift - r\sin dB + B\sin dr + B\cos rd\theta + \delta rd\theta\\
d\betaW^{\epsilon, \delta} &= d\betaWDLift + \delta dr\wedge d\theta,\\
X_{\betaW^{\epsilon, \delta}} &= X_{\betaWDLift} + r\sin X_{B}^{d\betaWDLift} + r\partial_{r} + B \delta^{-1}\left(r\cos \partial_{r} - \sin \partial_{\theta} \right).
\end{aligned}
\end{equation}
Here $\theta$ is the argument of each $\cos, \sin$ and $X^{d\betaWDLift}_{B}$ is the Hamiltonian vector field of a function $B \in \Cinfty(\Wdivisor)$ with respect to $d\betaWDLift$.

\subsection{Patching neighborhoods together}

To define the cobordism of Theorem \ref{Thm:FiberSumCobordism}, we seek to patch two $\overline{\disk^{\ast}}\times \Wdivisor$ neighborhoods together along an overlapping region. Define a radius $\epsilon$ annulus and associated half-annuli by
\begin{equation*}
\annulus = [-\epsilon, \epsilon]_{p} \times \Circle_{q}, \quad \annulus^{\pm} = \{ \pm p > 0\} \subset \annulus.
\end{equation*}
Define embeddings
\begin{equation*}
\Theta^{\pm}: \annulus^{\pm} \times \Wdivisor \rightarrow \overline{\disk^{\ast}} \times \Wdivisor, \quad (p, q, z) \mapsto (r=\pm p, \theta=\pm q, z)
\end{equation*}
for which we compute
\begin{equation}\label{Eq:BetaExtentionOverAnnulus}
\begin{aligned}
(\Theta^{\pm})^{\ast}\betaW^{\epsilon, \delta} &= \betaWDLift - p\sin dB + B\sin dp + B\cos pdq + \delta qdq.
\end{aligned}
\end{equation}
Here $q$ is the argument of each $\cos$ and we have used that $\cos$ is even to make all of the signs agree. Since the expression is independent of $\pm$ sign, $\betaW^{\epsilon, \delta}$ extends smoothly to a Liouville form over $\annulus \times \Wdivisor$.

\begin{defn}
By attaching $\annulus \times \Wdivisor$ to $W \setminus (\Wdivisor^{l} \sqcup \Wdivisor^{r})$ we define our cobordism $W_{\#}$ and write $\beta_{\#}$ for the Liouville form, which restricts to $(\Theta^{\pm})^{\ast}\betaW^{\epsilon, \delta}$ along the overlapping region.
\end{defn}

\begin{figure}[h]
\begin{overpic}[scale=.5]{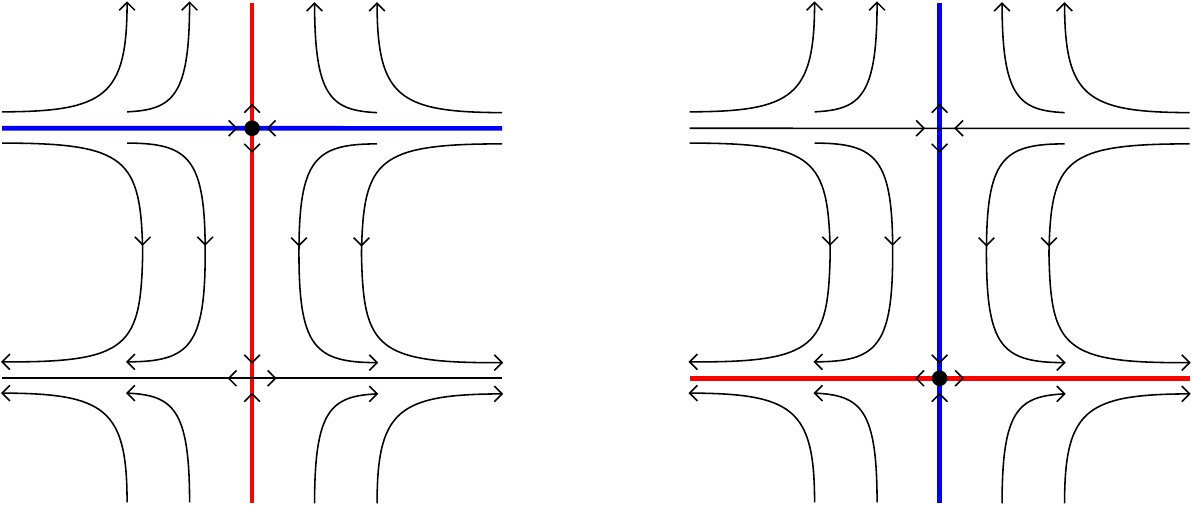}
\put(45, 9){$q=0$}
\put(45, 30){$q=\pi$}
\end{overpic}
\caption{Two depictions of the vector field $p\partial_{p} + \delta^{-1}(p\cos \partial_{p} - \sin \partial_{q})$ for $\delta > 0$ small. On the left, the zero $(0, \pi)$ is indicated as a black dot with its stable manifold in blue and the closure of its unstable manifold in red. On the right, the closure of the stable manifold and unstable manifold of $(0, 0)$ are similarly highlighted.}
\label{Fig:StableUnstableVF}
\end{figure}

To compute the Liouville vector field for $\beta_{\#}$ over $\annulus \times \Wdivisor$ we apply Equation \eqref{Eq:BetaModifiedNearZero}, to get
\begin{equation}\label{Eq:LiouvilleAnnulus}
\begin{aligned}
X_{\beta_{\#}} &= X_{\betaWDLift} + p\sin X_{B}^{d\betaWDLift} + p\partial_{p} + B\delta^{-1}\left(p\cos \partial_{p} - \sin \partial_{q} \right)
\end{aligned}
\end{equation}
where $q$ is now the argument of $\cos$ and $\sin$. We see that
\begin{equation}\label{Eq:BetaDeltaEpsilonZeros}
\left\{ (p, q, z)\ : \beta_{\#}|_{(p, q, z)}=0\right\} = \left\{ (0, 0, \zeta), (0, \pi, \zeta)\ :\ \betaWD|_{\zeta}=0\right\}.
\end{equation}
The projection of $X_{\beta_{\#}}$ is shown in Figure \ref{Fig:StableUnstableVF}.

\subsection{Lyapunov functions on the surgered cobordism}\label{Sec:LyapunovExtension}

While our construction of the Liouville form $\beta_{\#}$ is complete, we still need to discuss Lyapunov functions in the Weinstein case. For the remainder of the section assume that $(\Wdivisor, \betaWD)$ is Weinstein with Lyapunov function $F_{\Wdivisor}$. This implies that $F_{\Wdivisor}$ is Lyapunov for $\betaWD = e^{F_{\Wdivisor}}\betaWD$.

We must first extend the Weinstein structure over the overlapping region $\annulus \times \Wdivisor$. We recall from Lemma \ref{Lemma:PsiTaylor} that on $\disk_{x,y} \times \Wdivisor$ in the $\Psi$ coordinate system we have
\begin{equation*}
F_{W} = F_{\Wdivisor} + F_{x}x + F_{\hot}, \quad F_{x} = 1-\beta_{0}(X_{F_{\Wdivisor}}), \quad F_{\hot} \in \bigO(r^{2}).
\end{equation*}
The Lyapunov condition along $\{r = 0\}$ is determined by the $\bigO(r)$ expansion
\begin{equation}\label{Eq:FLyapunovInquality}
dF_{W}(X_{\betaW})|_{r = 0} = dF_{V}(X_{\betaWDLift}) + BF_{x} > 0.
\end{equation}
Here we get $>0$ due to the fact that both $dF_{W}$ and $X_{\betaW}$ are everywhere non-vanishing, since we are in a symplectization. Switching to polar coordinates and applying $\Theta^{\pm}$ we have
\begin{equation*}
(\Theta^{\pm})^{\ast}F_{W} = F_{\Wdivisor} + F_{x}|p| \cos + F_{\annulus, \hot}, \quad F_{\annulus, \hot} \in \bigO(p^{2})
\end{equation*}

The function does not extend smoothly over the annulus. Therefore we let $\rho_{1}^{+}(p), \rho_{1}^{0}$ be smoothings of $|p|$ with domain $\R$ such that there exists $C_{1}, C_{2} \in \R_{> 0}$ for which
\be
\item both are even functions with $\rho_{1}^{+}(p) = \rho^{0}_{1}(p)= |p|$ for $|p| > 1$,
\item both $\rho_{1}^{+}$ and $\rho^{0}_{1}$ have a unique minimum critical point at $0$,
\item restricting to each $p > 0$, the functions
\begin{equation*}
\rho^{+}_{1} - \rho^{-}_{0}\beta_{\Wdivisor}(X_{F_{\Wdivisor}}),\quad \frac{\partial \rho^{+}_{1}}{\partial p} - \frac{\partial \rho^{-}_{0}}{\partial p}\beta_{\Wdivisor}(X_{F_{\Wdivisor}})
\end{equation*}
in $\Cinfty(\Wdivisor)$ are bounded below point-wise by $C_{1}$.
\item $\rho_{1}^{+}(0) = C_{2}> 0$ and $\rho_{1}^{0}(0) = 0$.
\ee
Then we define
\begin{equation}\label{Eq:RhoProperties}
\begin{gathered}
\rho_{\epsilon}^{+}(p) = \epsilon\rho_{1}\left( \frac{p}{\epsilon}\right), \quad \rho_{\epsilon}^{0}(p) = \epsilon\rho_{1}^{0}\left( \frac{p}{\epsilon}\right)\\
\implies \sgn(p)\left(\rho^{+}_{1} - \rho^{-}_{0}\beta_{\Wdivisor}(X_{F_{\Wdivisor}})\right), \sgn(p)\left(\frac{\partial \rho^{+}_{1}}{\partial p} - \frac{\partial \rho^{-}_{0}}{\partial p}\beta_{\Wdivisor}(X_{F_{\Wdivisor}})\right) > \epsilon C_{1}, \quad \rho^{+}_{\epsilon}(0) = \epsilon C_{2}.
\end{gathered}
\end{equation}
Recalling that $F_{x} = 1 - \betaWD(X_{F_{\Wdivisor}})$ we define
\begin{equation}\label{Eq:LyapunovExension}
F_{\#}^{\epsilon} = F_{\Wdivisor} + (\rho_{\epsilon}^{+} - \rho_{\epsilon}^{0} \betaWD(X_{F_{\Wdivisor}}))\cos + h^{+}_{\epsilon}F_{\annulus, \hot}.
\end{equation}
Then for $|p| \leq \epsilon$,
\begin{equation*}
\begin{aligned}
F_{\#}^{\epsilon} &= F_{\Wdivisor} + (\rho_{\epsilon}^{+} - \rho_{\epsilon}^{0} \betaWD(X_{F_{\Wdivisor}}))\cos,\\
dF_{\#}^{\epsilon} &= dF_{\Wdivisor} + \cos d(\rho_{\epsilon}^{+} - \rho_{\epsilon}^{0} \betaWD(X_{F_{\Wdivisor}})) - (\rho_{\epsilon}^{+} - \rho_{\epsilon}^{0} \betaWD(X_{F_{\Wdivisor}})) \sin dq,\\
dF_{\#}^{\epsilon}|_{p=0} &= dF_{\Wdivisor} -\frac{\epsilon}{2} \sin dq.
\end{aligned}
\end{equation*}

Applying Equation \eqref{Eq:LiouvilleAnnulus}, we obtain
\begin{equation*}
\begin{aligned}
dF_{W}^{\epsilon}(X_{\betaW^{\epsilon, \delta}}) &= \left( dF_{\Wdivisor} - \rho_{\epsilon}^{0} \cos d(\betaWD(X_{F_{V}}))\right)\left(X_{\betaWDLift} +p\sin X^{d\betaWDLift}_{B} \right)\\
&+ p\left(1 + \delta^{-1} B\cos \right) \cos \left(\frac{\partial \rho_{\epsilon}^{+}}{\partial p} - \frac{\partial \rho_{\epsilon}^{0}}{\partial p}\betaWD(X_{F_{\Wdivisor}})\right)\\
&+ \delta^{-1}B\sin^{2}\left(\rho_{\epsilon}^{+} - \rho_{\epsilon}^{0} \betaWD(X_{F_{\Wdivisor}})\right)
\end{aligned}
\end{equation*}
which we seek to show is positive away from the zeros of $\beta_{\#}$ described in Equation \eqref{Eq:BetaDeltaEpsilonZeros}. From the definitions of $\rho_{\epsilon}^{+}$ and $\rho_{\epsilon}^{0}$, the third line above is below point-wise by $\epsilon\delta^{-1}\sin^{2}C_{1}$ where $C_{0}$ is the constant from Equation \eqref{Eq:Bbound}. Analyzing portions of the second line, we have
\begin{equation*}
p B\cos^{2}\left(\frac{\partial \rho_{\epsilon}^{+}}{\partial p} - \frac{\partial \rho_{\epsilon}^{0}}{\partial p}\betaWD(X_{F_{\Wdivisor}})\right) > |p|\cos^{2}C_{1}
\end{equation*}
from Equation \eqref{Eq:RhoProperties}. Therefore
\begin{equation*}
\begin{gathered}
dF_{W}^{\epsilon}(X_{\betaW^{\epsilon, \delta}}) > \frac{\epsilon}{\delta}(\sin^{2} + |p|\cos^{2}) + dF_{\Wdivisor}(X_{\betaWD}) + \hot\\
\hot = p\sin dF_{\Wdivisor}(X^{d\betaWDLift}_{B}) - \rho_{\epsilon}^{0}\cos d(\betaWD(X_{F_{\Wdivisor}}))\left( X_{\betaWDLift} + p\sin X^{d\betaWDLift}_{B}\right) + p \cos \left(\frac{\partial \rho_{\epsilon}^{+}}{\partial p} - \frac{\partial \rho_{\epsilon}^{0}}{\partial p}\betaWD(X_{F_{\Wdivisor}})\right)
\end{gathered}
\end{equation*}
With $\delta \ll \epsilon$ and using $\rho^{0}_{\epsilon}(0) = 0$ we see that the leading terms in $dF_{W}^{\epsilon}(X_{\betaW^{\epsilon, \delta}})$ dominate the higher order terms along the subset $\{ p \neq 0\}$. Then restricting attention to $\{ p=0\}$ we get
\begin{equation*}
dF_{W}^{\epsilon}(X_{\betaW^{\epsilon, \delta}})|_{p=0} = dF_{\Wdivisor}(X_{\betaWDLift}) + \frac{\epsilon}{2\delta}\sin^{2}B.
\end{equation*}
By the facts that $B$ is strictly positive and that $F_{\Wdivisor}$ is Lyapunov for $\betaWDLift$, $dF_{\#}^{\epsilon}(X_{\beta_{\#}})$ vanishes only at the zeros of $X_{\beta_{\#}}$ and is elsewhere positive. So $F^{\epsilon}_{\#}$ is Lyapunov for $\beta_{\#}$ and our cobordism is Weinstein.

\subsection{Stable and unstable manifolds of $X_{\beta_{\#}}$}\label{Sec:UnstableAnalysis}

We seek to understand the stable and unstable manifolds associated to the critical points $p=0$, $q\in \{0, \pi\}$, $\zeta \in \WdivisorLift$ of $X_{\beta_{\#}}$ where $\zeta$ is a zero of $X_{\betaWD}$. For notational simplicity we work under the assumption that there is a single such $\zeta$ so that we have two critical points
\begin{equation*}
\thicc{\zeta}^{-} = (0, \pi, \zeta), \quad \thicc{\zeta}^{+} = (0, 0, \zeta).
\end{equation*}

At the critical points, we use $F_{x}|_{\zeta}=1$ from Lemma \ref{Lemma:PsiTaylor} to compute their critical values as
\begin{equation*}
F^{\epsilon}_{\#}(\thicc{\zeta}^{\pm}) = F_{\Wdivisor}(\zeta) \pm \frac{\epsilon}{2}.
\end{equation*}
Looking at Equation \eqref{Eq:LiouvilleAnnulus}, we obtain
\begin{equation}\label{Eq:UnstableExplicit}
\begin{aligned}
\Stable(\thicc{\zeta}^{-}, X_{\beta_{\#}}) \cap \left(\annulus \times \Wdivisor \right) &= \{ (p, \pi, z)\ : z \in \Stable(X_{\betaWD})\}\\
\Unstable(\thicc{\zeta}^{-}, X_{\beta_{\#}}) \cap \left(\annulus \times \Wdivisor \right) &= \{ (0, q, z)\ : q\neq 0, z \in \Unstable(X_{\betaWD})\}\\
\Stable(\thicc{\zeta}^{+}, X_{\beta_{\#}}) \cap \left(\annulus \times \Wdivisor \right) &= \{ (0, q, z)\ : q\neq \pi, z \in \Stable(X_{\betaWD})\}\\
\Unstable(\thicc{\zeta}^{+}, X_{\beta_{\#}}) \cap \left(\annulus \times \Wdivisor \right) &= \{ (p, 0, z)\ : z \in \Unstable(X_{\betaWD})\}.
\end{aligned}
\end{equation}
That is, the stable and unstable manifolds intersect our overlap region $\annulus \times \Wdivisor$ in products of
\be
\item stable and unstable manifolds of the vector field $p(1 + \delta^{-1}\cos)\partial_{p} - \delta^{-1}\sin \partial_{q}$ with
\item the stable and unstable manifolds of $X_{\betaWD}$.
\ee
Figure \ref{Fig:StableUnstableVF} provides all of the intuition here. It follows that
\begin{equation*}
\ind (\zeta^{\pm}) = \ind (\zeta) + 1.
\end{equation*}

So we see that the convex end of our cobordism is the result of two Weinstein handle attachments. First we attach a handle to the concave boundary $\Mxi$ along its intersection with $\Stable(\thicc{\zeta}^{-}, X_{\beta_{\#}})$, which is isotropic. Second, we attach a Weinstein handle to the level set $\{ F^{\epsilon}_{\#} = F_{\Wdivisor}(\zeta) \}$ along its intersection with $\Stable(\thicc{\zeta}^{-}, X_{\beta_{\#}})$. When $\zeta$ is of critical index $n$ in $\Wdivisor$, then $\thicc{\zeta}$ is of critical index $n+1$ in our cobordism and each Weinstein surgery described above has the effect of a contact $-1$ surgery on contact level sets of $F^{\epsilon}_{\#}$. The proof of Theorem \ref{Thm:FiberSumCobordism} is complete.

\subsection{Proof of Lemmas \ref{Lemma:FiberSumOT} and \ref{Lemma:FlexibleFiberSum}}

In the critical index case, we can alternatively describe the concave end of our cobordism as the result of a pair of contact anti-surgeries. Here the analysis is essential for our application to Theorem \ref{Thm:MainOT}. First we apply a contact anti-surgery to the intersection of the convex end of the cobordism along its intersection $\Leg^{+}$ with $\Unstable(\thicc{\zeta}^{+}, X_{\beta_{\#}})$. Then we apply a contact anti-surgery to the intersection of the level set $\{ F^{\epsilon}_{\#} = F_{\Wdivisor}(\zeta) \}$ along its intersection $\Leg^{-}$ with $\Unstable(\thicc{\zeta}^{+}, X_{\beta_{\#}})$. The $\Leg^{\pm}$ are Legendrian spheres.

To get an understanding of $\Leg^{+}$, we observe that $\{\pm p \geq 0\} \subset \Unstable(\thicc{\zeta}^{+}, X_{\beta_{\#}})$, as described above, extends beyond its intersection with $\annulus \times \Wdivisor$ as the $\Unstable(\Unstable(\zeta, X_{\betaWD}), X_{\betaW})$ associated to $\Wdivisor_{i} \subset M$, described in the end of Lemma \ref{Lemma:PsiTaylor}. This is because when switching from polar to Cartesian coordinates, $\{ \pm p > 0, q=0 \}$ are sent to $\{ x > 0, y=0\}$ in the $\disk_{x, y} \times \Wdivisor$ coordinate systems. The $\Unstable(\Unstable(\zeta, X_{\betaWD}), X_{\betaW})$ associated to the $\Wdivisor_{i} \subset \Mxi$ intersect the convex end of our cobordism -- a fiber sum along the $\divSet_{i}^{+}$ -- in the $\Unstable(\zeta, X_{\betaWD}) \subset \Wdivisor_{i}$, which we can view as being contained in the positive end of our cobordism. Since we are assuming that at least one of the $\Unstable(\zeta, X_{\betaWD})$ is stablilized, then $\Leg^{+}$ will be as well. Indeed a stabilized chart is assumed to be disjoint from the positive boundary $\divSet^{+}_{i}$ of one of the $\Wdivisor_{i}$, which misses the fiber sum surgery locus. So we can view the stabilized chart as being contained in the positive boundary of our cobordism.

The above discussion entails that the contact level set $\{ F^{\epsilon}_{\#} = F_{\Wdivisor}(\zeta) \}$ of our cobordism is overtwisted. Since the convex end of our cobordism $\Mxi$ is obtained from a contact anti-surgery on this level set, we conclude that $\Mxi$ is overtwisted as well from Lemma \ref{Lemma:PlusOneSurgeryPreservesOT}.

The proof of Lemma \ref{Lemma:FlexibleFiberSum} follows from an identical analysis of stable manifolds of the critical points of $F_{\#}$ on $W_{\#}$ using Equation \eqref{Eq:UnstableExplicit}.

\subsection{Proof of Theorem \ref{Thm:FiberSumOT}}

Assume the hypothesis of Theorem \ref{Thm:FiberSumOT}. Using Lemma \ref{Lemma:HandleLiouvilleExtension} we can assume that $\phi_{0}\divSet$ is the convex boundary of a Weinstein hypersurface $\Wdivisor_{0}$ with a single $\ind=n$ handle $H_{n}$ which is stabilized in the complement of $\phi_{1}\divSet$. Find any handle in the complement of $\Wdivisor_{0}$ giving Weinstein surgery data for $\phi_{1}\divSet$ and compatible with $\framing_{1}$ so that we can realize $\phi_{1}\divSet$ as the concave boundary of some $\Wdivisor_{1} \subset M$.

The $\Wdivisor_{i}$ both have the same Weinstein structures so that we can apply Theorem \ref{Thm:FiberSumCobordism} to find a Weinstein cobordism whose concave boundary is the fiber sum $(M_{\#}, \xi_{\#})$ along the $(\phi_{i}\divSet, \framing_{i})$ and whose convex boundary is the fiber sum along the $\partial^{+}\Wdivisor_{i}$. Since the handle $H_{n}$ is stabilized, Lemma \ref{Lemma:FiberSumOT} tells us that $(M_{\#}, \xi_{\#})$ is overtwisted, completing the proof.

\section{Proofs of the main theorems}\label{Sec:MainProofs}

Here we use the results we've collected throughout the body of the article to prove the theorems stated in the introduction. To start, we'll need an explicit description of the fiber sum of a contact manifold along a pair of unknots.

\subsection{Fiber sums of standard unknots}\label{Sec:FiberSumUnknots}

Consider disjoint Weinstein embeddings $\phi_{1},\phi_{2}$ of the standard Weinstein disk into a connected $\Mxi$ of $\dim=2n+1$. In the notation of \S \ref{Sec:FiberSum}, $\divSet^{-}_{i} = \emptyset$, the $\divSet_{i}^{+}$ are standard unknots by definition, and $F_{\Wdivisor}$ has a single critical point of $\ind = 0$. To obtain such a disjoint pair, we may as well start from a single $(\Wdivisor, \beta) = (\disk^{2n}, \beta_{std}) \subset \Mxi$ and take the push-off $\Wdivisor_{\uparrow}$ for the second copy as in Lemma \ref{Lemma:DisjointUnknots}.

For the cobordism $(W_{\#}, \beta_{\#})$ of Theorem \ref{Thm:FiberSumCobordism}, we get $(M^{-}_{\#}, \xi_{\#}^{-})=\Mxi$ for the concave boundary and $(M^{+}_{\#}, \xi^{+}_{\#})$ -- the contact fiber sum along the disjoint copies of $\unknot$ -- for the convex boundary. The Weinstein cobordism is obtained by attaching a pair of $\ind = 1$ Weinstein handles to a finite symplectization of $\Mxi$. Since $M$ is connected, we can assume that the attaching loci of the handles -- which are arbitrarily small Darboux disks -- are contained in a Darboux ball in $\Mxi$ to obtain
\begin{equation}\label{Eq:FiberSumUnknots}
(M^{+}_{\#}, \xi^{+}_{\#}) = \Mxi \# 2(\Circle \times \sphere^{2n}, \xi_{std}).
\end{equation}
Here $(\Circle \times \sphere^{2n}, \xi_{std})$ is the standard contact structure, which can be seen as the boundary of the Weinstein manifold obtained by rounding the corners of $([-1, 1]_{p}\times \Circle_{q} \times \disk^{2n}_{z=(x + iy)}, pdq + \beta_{std}, p^{2} + \norm{z}^{2})$.

\begin{lemma}\label{Lemma:KillOneHandleOT}
In the above notation, there is a Weinstein cobordism $(W, \betaW, F_{W})$ with concave boundary $(M^{+}_{\#}, \xi^{+}_{\#})$ and convex boundary $\Mxi$. When $\dim M \geq 5$ this cobordism is subcritical, so if $(M^{+}_{\#}, \xi^{+}_{\#})$ is overtwisted, then $\Mxi$ is overtwisted as well by Lemma \ref{Lemma:SubcriticalPreservesOT}.
\end{lemma}

\begin{proof}
This follows immediately from the fact that there is a Weinstein cobordism having $(\Circle \times \sphere^{2n}, \xi_{std})$ as its concave boundary and $\stdSphere{2n+1}$ as its convex boundary, and that such a cobordism can be assumed subcritical when $n \geq 2$. We can view $(\Circle \times \sphere^{2n}, \xi_{std})$ as the boundary of a tubular neighborhood of a circle in $(\disk^{2n+2}, \beta_{std})$. Attaching a $\ind=2$ Weinstein handle, which is subcritical in the $\dim M \geq 5$ setting, we get back $\stdSphere{2n+1}$.
\end{proof}

\subsection{Unknots cannot be stabilized in tight $\Mxi$}

We are now ready to complete the proof of Theorem \ref{Thm:MainOT}. Provided the results in \S \ref{Sec:Stabilization}, it remains to show that if $\unknot \subset \Mxi$ is a stabilization, then $\Mxi$ is overtwisted. So we suppose that $\unknot$ is a stabilization.

Applying Lemma \ref{Lemma:HandleLiouvilleExtension}, we can find a Liouville hypersurface $(\Wdivisor, \betaWD)$ with $\partial^{-}\Wdivisor = \unknot$ having a single critical point of index $n$ whose unstable manifold $\Unstable_{n}$ is stabilized in a cusp chart $\disk^{2n+1}$ in the complement of $\divSet^{+} = \partial^{+}\Wdivisor$. We also have a ``destabilized'' $(\Wdivisor', \betaWD)$ for which $\partial^{-}\Wdivisor' = \unknot$ with $\Wdivisor'$ and $\Wdivisor$ agreeing outside of the cusp chart $\disk^{2n+1}$. The destabilized $\Wdivisor'$ also has a single critical point of index $n$ with associated unstable manifold $\Unstable'$. So $\Unstable$ is a stabilization of $\Unstable'$ within a cusp chart.

\begin{figure}[h]
\begin{overpic}[scale=.4]{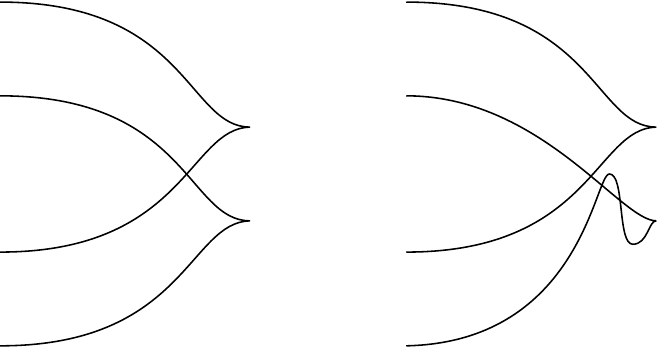}
\end{overpic}
\caption{On the left, a dimensionally-reduced depiction of $\Unstable'$ (bottom) and $\Unstable'_{\uparrow}$ (top) shown as $n=1$ Legendrian fronts. On the right, analogous depictions of $\Unstable$ (bottom) and $\Unstable'_{\uparrow}$ (top).}
\label{Fig:TinyStab}
\end{figure}

Consider a push-off $\Wdivisor'_{\uparrow}$ of the destabilized $\Wdivisor'$ with associated unstable manifold $\Unstable'_{\uparrow}$. Then $\Unstable'$ and $\Unstable'_{\uparrow}$ will be disjoint. Inside using the model for stabilization from \S \ref{Sec:LegStabilization} and using the $(z, p, q)$ coordinated there, we will have when restricting a sufficiently small neighborhood of the cusp, $\inf(z|_{\Unstable'_{\uparrow}}) < \sup(z|_{\Unstable'})$. So we can assume that $\Unstable$ is obtained by performing a Legendrian stabilization within a cusp chart contained in $\disk^{2n+1}$ which is disjoint from $\Unstable'_{\uparrow}$. By making the $\ind=n$ Weinstein handles of $\Wdivisor'_{\uparrow}$ and $\Wdivisor$ thin, we can then assume that $\Unstable$ is stabilized in the complement of the union of $\divSet^{+}$ with $\Wdivisor'_{\uparrow}$. Figure \ref{Fig:TinyStab} provides a schematic.

Applying Lemma \ref{Lemma:DisjointUnknots} and Equation \eqref{Eq:FiberSumUnknots}, the fiber connect sum of $\unknot$ with $\unknot' = \partial^{-}\Wdivisor'$ is $\Mxi \# 2(\Circle \times \sphere^{2n}, \xi_{std})$. By Theorem \ref{Thm:FiberSumOT}, this contact manifold is overtwisted. Applying Lemma \ref{Lemma:KillOneHandleOT}, $\Mxi$ is overtwisted as well, completing the proof of Theorem \ref{Thm:MainOT}.

\subsection{Liouville boundaries}\label{Sec:HypersurfaceProof}

Now we prove Theorem \ref{Thm:HypersurfaceNonSimple}. So we suppose that $(\Wdivisor, \betaWD)$ is a Liouville hypersurface in some $\Mxi$ with $\partial^{-}\Wdivisor = \emptyset$ and $\partial^{+}\Wdivisor = \divSet$, the contact divisor of interest. We seek to show that $\divSet$ cannot be a stabilization if $\Mxi$ is weakly fillable or has non-vanishing contact homology.

Take a second copy of $(M', \xi)$ of $\Mxi$ within which we have another copy $\divSet'$ of $\divSet$ bounding a second copy $\Wdivisor'$ of $\Wdivisor$. Theorem \ref{Thm:FiberSumCobordism} tells us that there is a Liouville cobordism $(W_{\#}, \beta_{\#})$ whose concave boundary is $\Mxi \sqcup \Mxi$ and whose convex boundary is a fiber sum along the framed contact divisors $(\divSet, \framing_{\Wdivisor}), (\divSet', \framing_{\Wdivisor'})$. Theorem \ref{Thm:FiberSumOT} tells us that if $\divSet$ is stabilized, then the convex boundary $\partial^{+}W_{\#}$ of the cobordism is overtwisted.

We first address contact homology, $CH$ \cite{BH:ContactDefinition, Pardon:Contact}, since the proof is very easy. If $\Mxi$ has non-vanishing contact homology $CH\Mxi \neq 0$ with $\Q$ coefficients, then
\begin{equation*}
CH(\Mxi \sqcup \Mxi) \simeq CH\Mxi^{\otimes 2} \neq 0.
\end{equation*}
The convex boundary of $W_{\#}$ must have non-vanishing $CH$ as well, since $W_{\#}$ induces a unital algebra morphism from the $CH$ of the convex boundary to the $CH$ of the concave boundary. On the other hand, if $\divSet$ was stabilized, then the $CH$ of the convex boundary would be zero, since the $CH$ of overtwisted contact manifolds vanishes by \cite{BvK:Stabilization}.

By \cite{LW:Torsion, NW:WeakFillings} symplectic fillability implies the non-vanishing of $CH$ with $\Q$ coefficients, and weak fillability guarantees the non-vanishing of $CH$ with fully twisted $\Q[H_{2}(M)]$ coefficients. Hence the case of $\Mxi$ (strongly) symplectically fillable has already been covered. Note that the argument of the preceding paragraph breaks down for twisted $\Q[H_{2}(M)]$ coefficients: Unless the inclusions $H_{2}(\partial^{\pm}W_{\#}) \rightarrow H_{2}(W_{\#})$ are isomorphisms, we won't be able to associate a unital algebra morphism $CH(\partial^{+}W_{\#}) \rightarrow CH(\partial^{-}W_{\#})$ to our cobordism when using twisted coefficients.

So we now assume that $\Mxi$ is weakly symplectically fillable. We'll establish that $\partial^{+}W_{\#}$ is weakly fillable using a Morse-theoretic approach. From this is will follow that $\partial^{+}W_{\#}$ is tight and so $\divSet$ is not stabilized by Theorem \ref{Thm:FiberSumOT}.

Choose a lifting function $F_{\Wdivisor}$ for $\Wdivisor$ and consider the involution $\tau$ of $M \sqcup M$ which swaps the factors of the disjoint union. This is clearly a contact transformation which extends to a Weinstein automorphism of the symplectization of $M$ sending the lift $\WdivisorLift$ of $\Wdivisor$ in the first factor to the lift in the second factor, and vice versa. We claim that $\psi$ extends to an involution $T$ of the Liouville cobordism $W_{\#}$. To do so, we note that it must take the form $T(p, q, z) = (-p, -q, z)$ in the ``overlapping region'' $[-\epsilon, \epsilon]_{p} \times \Circle_{q} \times \Wdivisor$ of the cobordism. According to Equation \eqref{Eq:BetaExtentionOverAnnulus} this involution preserves the Liouville form on $W_{\#}$. So $T$ is a Liouville isomorphism of $(W_{\#}, \beta_{\#})$ for which $T|_{\partial^{-}W_{\#}} = \tau$.

Take a weak filling $(W_{\downarrow}, \omega_{\downarrow})$ of $\Mxi$. Then $\tau$ extends over $W_{\downarrow} \sqcup W_{\downarrow}$ in the obvious fashion, giving us a symplectic automorphism of $(W_{\downarrow} \sqcup W_{\downarrow}, \omega_{\downarrow}\sqcup\omega_{\downarrow})$. We can then smoothly glue $\partial^{-}W_{\#}$ to $\partial^{+}W_{\downarrow}$ to obtain a smooth filling $W$ of $\partial^{+}W_{\#}$. Then $T$ is defined over all of $W$. The fixed point locus of $T$ is the set
\begin{equation*}
\Fix(T) = \{ p = 0, q \in \{0, \pi\}\} \subset [-\epsilon, \epsilon] \times \Circle \times \Wdivisor \subset W.
\end{equation*}

Even if $\Wdivisor$ is not Weinstein, we can assume that the lifting function $F_{\Wdivisor}$ on $\Wdivisor$ is Morse. Then analysis of \S \ref{Sec:LyapunovExtension} equips $W_{\#}$ with a Morse function $F_{\#}$ which extends the symplectization coordinate $s$ near $\partial^{-}W_{\#}$. By Equation \eqref{Eq:LyapunovExension}, this function is invariant under the involution $T$. The analysis of \S \ref{Sec:UnstableAnalysis} then describes the stable and unstable manifolds of $F_{\#}$ (even when $F_{\Wdivisor}$ is not a Lyapunov function) by replacing $X_{\betaWD}$ with $\grad F_{\Wdivisor}$ throughout if we assume that the gradient is defined using a Riemannian metric $g$ which is invariant under $\Psi$ and takes the form $g = dp^{\otimes 2} + dq^{\otimes 2} + g_{\Wdivisor}$ for a metric $g_{\Wdivisor}$ on $\Wdivisor$ along the overlapping region.

For each $\ind=m$ critical point $\zeta \in \Wdivisor$ of $F_{\Wdivisor}$, there are two critical points $\thicc{\zeta}^{\pm}$ of $F_{\#}$ at $(0, \pi, \zeta)$ and $(0, 0, \zeta)$ in the overlapping region $[-\epsilon, \epsilon]_{p} \times \Circle_{q} \times \Wdivisor$. So the critical points are all contained in $\Fix(T)$. From Equation \eqref{Eq:UnstableExplicit}, $T$ restricts to an orientation reversing diffeomorphism of the stable manifolds of these critical points.

It follows that $T$ induces an involution of each level set $F_{\#}^{-1}(s) \subset W_{\#} \subset W$ of $F_{\#}$, inducing an orientation reversing diffeomorphism of the intersection of regular level set with the aforementioned stable manifolds. These will be the attaching spheres $\sphere^{\ind(\zeta)}(\thicc{\zeta}^{\pm}, s)$ of a Morse handle-body decomposition of the cobordism. The stated dimensions of the spheres follows from $\ind(\thicc{\zeta}^{\pm}) = \ind(\zeta) + 1$.

Extending the collar neighborhood of $\partial W_{\downarrow}$ described in \S \ref{Sec:SympManifolds}, we can assume that the collar has the form $[0, C]_{s} \times M$ along which $\omega_{\downarrow} = d(e^{s}\alpha) + \omega_{M}$ with $\partial W_{\downarrow} = \{C\} \times M$. The important point here is that by making $C \gg 0$, the derivative of Liouville term $e^{s}\alpha$ dominates $\omega_{M}$ point-wise. When we attach the concave end of $W_{\#}$ to the convex end of $W_{\downarrow}$, $d(e^{s}\alpha)$ extends over $W_{\#}$ as $e^{C}d(\beta_{\#})$. If we can ensure that $\omega_{M}$ extend over $W_{\#}$, then we will have a symplectic form on all of $W$ by taking $C \gg 0$. 

The obstruction to extending $\omega_{M, 0} = \omega_{M}\sqcup \omega_{M}$ -- defined on $\partial^{-}W_{\#} = M \sqcup M$ -- to a closed form over all of $W_{\#}$ is determined entirely by the boundaries of the $\ind=3$ handles. We can always extend $\omega_{M, 0}$ to a closed $\omega_{M, 1}$ over the $1$-handles of $W_{\#}$, and can assume -- by averaging over $T$ -- that this extension is $T$-invariant. Then we extend to a closed $\omega_{M, 2}$ over the $2$-handles in a similar fashion. When we try to attach the $3$-handles, there is an obstruction $\bigO_{2}$ to extension given by
\begin{equation*}
\bigO_{2} = \int_{\sphere^{\ind(\zeta)}(\thicc{\zeta}^{\pm}, s)}\omega_{M, 2} \in \R.
\end{equation*}
The obstruction depends on a choice of orientation of each such sphere. Cohomological long exact sequences for the handle attachments tell us that an extension of $\omega_{M,2}$ to a closed $\omega_{M, 3}$ over the $3$-handles of $W_{\#}$ exists iff $\bigO_{2}=0$. Using the fact that $\Psi$ reverses orientations of our attaching spheres, we compute
\begin{equation*}
\bigO_{2} = \int_{\sphere^{\ind(\zeta)}(\thicc{\zeta}^{\pm}, s)}\Psi^{\ast}\omega_{M, 2} = \int_{-\sphere^{\ind(\zeta)}(\thicc{\zeta}^{\pm}, s)}\omega_{M, 2} = -\bigO_{2}.
\end{equation*}
so that the obstruction must vanish. There are no obstructions to extending over the handles of higher index.

So we have a closed $\omega_{\#} = \omega_{M, 2n}\in \Omega^{2}(W_{\#})$ which agrees with $\omega_{M}\sqcup\omega_{M}$ along its concave end. Therefore $e^{C}\beta_{\#} + \omega_{\#} \in \Omega^{2}(W_{\#})$ is a closed $2$-form extending $\omega_{\downarrow}\sqcup\omega_{\downarrow}$ defined along its concave boundary. For $C$ sufficiently large, this form is symplectic on $W_{\#}$ and will be fiber-wise symplectic on the contact plane field $\xi_{\#}^{+}$ of $\partial^{+} W_{\#}$. So following \cite{MNW13}, we have a weak symplectic filling of $(\partial^{+}W_{\#}, \xi_{\#}^{+})$  given by
\begin{equation*}
(W = W_{\#} \cup (W_{\downarrow} \sqcup W_{\downarrow}), \omega_{\#} \cup (\omega_{\downarrow} \sqcup \omega_{\downarrow})).
\end{equation*}
Here $\cup$ indicates union in all instances (rather than, say, a cup product).

Weak fillability implies tightness \cite{Gromov:JCurves, Klaus:Plastik}, so we conclude that if $\divSet$ is a contact divisor bounding a Liouville hypersurface in a weakly fillable $\Mxi$, then $\divSet$ cannot be stabilized. The proof of Theorem \ref{Thm:HypersurfaceNonSimple} is complete.

\subsection{Bindings of open books are not stabilized}

Now we prove Theorem \ref{Thm:BindingNonStabilized} whose first statement asserts that bindings of supporting books of contact manifolds are never stabilized.

Suppose that $\Mxi$ is presented as an (abstract) open book with Liouville page $(\Wdivisor, \betaWD)$ and monodromy $\phi \in \Symp(\Wdivisor, \partial \Wdivisor, d\betaWD)$. We view a single page $\Wdivisor$ as a Liouville hypersurface in $\Mxi$ whose boundary $\divSet$ is the binding of the open book. We seek to show that $\divSet$ is not stabilized.

As in the previous subsection, we consider two copies $\divSet, \divSet'$ of $\divSet$ in distinct copies of $\Mxi$, framed by pages of the open books. As described in \cite{Avdek:Liouville}, the fiber sum of $\divSet$ and $\divSet'$ defined using this framing is a bundle over $\Circle$ whose fibers are convex hypersurfaces $S$ with monodromy $S \rightarrow S$ determined by $\phi$. The convex hypersurfaces have positive and negative region equal to $(\Wdivisor, \beta)$ with their boundaries identified using $\Id_{\divSet}$ so long as we identify the pages of the open books using the identity map.

This convex hypersurface $S$ can be viewed as the boundary of the contactization $([-1, 1]_{t} \times \Wdivisor, \ker(dt + \betaWD))$ with its corners rounded \cite{Avdek:Liouville}. This contactization is tight, since it gives a neighborhood of a page of the open book with page $(\Wdivisor, \beta)$ and monodromy $\Id_{\Wdivisor}$, and such an open book is the boundary of the Liouville manifold obtained by rounding the corners of $(\disk^{2} \times \Wdivisor, xdy - ydx + \betaWD)$. Alternatively, one can use sutured $CH$ \cite{CGHH:Sutures} to establish tightness of a contactization. We conclude that $S$ has a tight neighborhood $\R \times S$ along which the contact structure is $\R$ invariant. This $\R \times S$ covers the bundle over $S$, so we conclude that the fiber sum over $\divSet, \divSet'$ with the specified framings is tight. If $\divSet$ was stabilized, then this fiber sum would be overtwisted by Theorem \ref{Thm:FiberSumOT}. So $\divSet$ cannot be stabilized and the first statement of Theorem \ref{Thm:BindingNonStabilized} is established.

Now we use this fact to show that any $\Mxi$ has infinitely many non-simple contact divisors. By \cite{Giroux:ContactOB, BHH:OB, Sackel:Handle}, every $\Mxi$ admits a supporting open book decomposition. Choose one and let $\divSet$ be the binding with page $\Wdivisor$ and monodromy $\phi$. We can positively stabilize the open book (cf. \cite{BHH:OB}) to obtain a new supporting open book as follows:
\be
\item The new page $\Wdivisor'$ is a boundary connected sum of $\Wdivisor$ and $\disk^{\ast}\sphere^{n}$, the latter equipped with its canonical Weinstein structure. This can be seen as the result of a Weinstein $1$-handle attachment to $\Wdivisor \sqcup \disk^{\ast}\sphere^{n}$.
\item The new monodromy is $\phi$ when restricted to $\Wdivisor \subset \Wdivisor'$ and a symplectic Dehn twist along the zero section of $\disk^{\ast}\sphere^{n} \subset \Wdivisor'$.
\ee

The new binding $\divSet'$ is topologically $\divSet\#(\sphere^{\ast}\sphere^{n})$. Since $H_{n-1}(\sphere^{\ast}\sphere^{n}, \Z/2)\simeq\Z/2$ -- generated by a cotangent sphere fiber -- it follows that
\begin{equation*}
\rank H_{n-1}(\divSet', \Z/2) = \rank H_{n-1}(\divSet, \Z/2) + 1.
\end{equation*}
Comparing these homology ranks using any number of stabilizations of our initial open book, we have infinitely many (topologically) inequivalent contact divisors. Using the fact that these bindings are not stabilized and taking contact stabilizations, eg. as in \S \ref{Sec:StandardStabilization}, we see that they are all non-simple, completing the proof of Theorem \ref{Thm:BindingNonStabilized}.

\subsection{Implicitly overtwisted divisors are stabilized in high dimensions}\label{Sec:OTDivisor}

We address Theorem \ref{Thm:OTDivisor} starting with the following prerequisite lemma.

\begin{lemma}
Suppose that $\divSet \subset \Mxi$ is a contact divisor such that $\dim M \geq 7$. If $\Leg \subset \Mxi$ is a Legendrian disk for which $\Leg \cap \divSet$ is a stabilized Legendrian sphere in $\MxiDivSet$, then $\Leg$ is Legendrian stabilized in the complement of $\divSet$.
\end{lemma}

\begin{proof}
We adapt the strategy of the proof of \cite[Lemma 4.4]{Eliash:WeinsteinRevisited} to the present context. So we will use loose charts \cite{Murphy:Loose}, rather than Legendrian stabilized charts, in contrast with the rest of this article.

Since $\MxiDivSet$ is overtwisted $\partial \Leg$ is a loose Legendrian sphere. Let $N_{\divSet} \subset \MxiDivSet$ be a loose chart for $\partial \Leg$ in $\MxiDivSet$ in which we use a contact form $\alpha_{\divSet}$ for $\xiDivSet$. Then we have a Weinstein neighborhood $N_{M} = \disk(\rho) \times N_{\divSet} \subset \Mxi$ along which we can assume
\begin{equation*}
\xi = \ker(\alpha_{\divSet} + pdq), \quad \Leg \cap N_{M} = \partial \Leg \times \{ p=0, q > 0\}
\end{equation*}
using a Weinstein neighborhood construction for some sufficiently small radius $\rho > 0$. It follows that $\{ q \in [\epsilon, 2\epsilon]\} \subset \Leg$ is the product of a loose chart with a zero section in a cotangent disk bundle for $\epsilon > 0$ small. As described in \cite{Eliash:WeinsteinRevisited}, it follows that $\{ q \in [\epsilon, 2\epsilon]\} \subset \Leg$ has a loose chart in the complement of $\divSet$, completing the proof of the lemma.
\end{proof}

Now we prove Theorem \ref{Thm:OTDivisor}. Suppose that $\divSet \subset \Mxi$ is as in the above lemma with $\MxiDivSet$ is overtwisted. As in the construction of the standard stabilization, we can find a Legendrian disk $\Leg \subset \Mxi$ for which $\partial \Leg = \Leg \cap \divSet$ is a standard $\dim=n-1$ Legendrian unknot and build an ambient Weinstein handle $H_{n}$ containing $\Leg$ as its stable manifold $\Stable$. Intrinsically $(\divSet^{+}, \xi_{\divSet^{+}})$ is the result of a contact $-1$ surgery along a Legendrian unknot $\Leg_{U}^{n-1} \subset \MxiDivSet$. Since this $\dim=n-1$ unknot is contained in a Darboux ball, $(\divSet^{+}, \xi_{\divSet^{+}})$, the result of the surgery is the result of a connected sum $\Mxi$ with $(\sphere^{\ast}\sphere^{n}, \xi_{std})$ -- the latter being the result of a contact $-1$ surgery on $\Leg^{n-1}_{U} \subset \stdSphere{2n-1}$. Hence $\xi_{\divSet^{+}}$ is overtwisted. The unstable manifold $\Unstable \subset H_{n}$ is Legendrian stabilized in the complement of $\divSet^{+}$ by the preceding lemma, so $\divSet$ is a stabilization and the proof of Theorem \ref{Thm:OTDivisor} is complete.

An identical method of proof establishes that $\divSet$ is an anti-stabilization as well.

\textsc{Universit\'{e} Paris-Saclay, Laboratoire de Math\'{e}matiques d'Orsay, Orsay, France}\par\nopagebreak
\textit{Email:} \href{mailto:russell.avdek@universite-paris-saclay.fr}{russell.avdek@universite-paris-saclay.fr}\par\nopagebreak
\textit{URL:} \href{https://www.russellavdek.com/}{russellavdek.com}

\end{document}